\documentclass{amsart}

\usepackage{amssymb}

\usepackage{stmaryrd}
\usepackage{amsfonts}
\usepackage{graphicx}
\usepackage{subfig}
\usepackage[nobysame]{amsrefs}
\DeclareMathOperator{\Span}{span}

\usepackage{theoremref}
\usepackage{hyperref}
\newtheorem{theorem}{Theorem}[section]
\newtheorem{lemma}[theorem]{Lemma}
\newtheorem{co}[theorem]{Corollary}
\newtheorem{prop}[theorem]{Proposition}

\theoremstyle{definition}
\newtheorem{definition}[theorem]{Definition}

\newtheorem{ass}[theorem]{Assumption}
\theoremstyle{remark}
\newtheorem{remark}[theorem]{Remark}

\numberwithin{equation}{section}

\begin{document}

\title[inf-sup stable quasi-reversibility for unique continuation]{On inf-sup stability and optimal convergence of the quasi-reversibility method for unique continuation subject to Poisson's equation}


\author{Erik Burman}
\address{Department of Mathematics, University College London, UK}
\curraddr{}
\email{e.burman@ucl.ac.uk}
\thanks{}
\author{Mingfei Lu}
\address{Department of Mathematics, University College London, UK}
\curraddr{}
\email{mingfei.lu.21@ucl.ac.uk}
\thanks{}


\subjclass[2020]{Primary }

\date{}

\dedicatory{}

\begin{abstract}
In this paper, we develop a framework for the discretization of a mixed formulation of quasi-reversibility solutions to ill-posed problems with respect to Poisson's equations. By carefully choosing test and trial spaces a formulation that is stable in a certain residual norm is obtained. Numerical stability and optimal convergence are established based on the conditional stability property of the problem. Tikhonov regularisation is necessary for high order polynomial approximation, but its weak consistency may be tuned to allow for optimal convergence. For low order elements a simple numerical scheme with optimal convergence is obtained without stabilization.   We also provide a guideline for feasible pairs of finite element spaces that satisfy suitable stability and consistency assumptions. Numerical experiments are provided to illustrate the theoretical results.
\end{abstract}

\maketitle

\section{Introduction}
A common situation when considering data assimilation problems or inverse problems is that the data available for the approximation of the partial differential equations is not such that the problem becomes well-posed. This is typically the case when data of the boundary condition is unknown on part of the boundary, but additional data is available elsewhere. There are two model situations. Either measurements of the solution are available in some subset of the bulk domain or, in the elliptic Cauchy problem, both Neumann and Dirichlet data are known somewhere on the boundary. In both cases one wish to continue the solution from the set where the data is available into the bulk domain. It is known that under certain assumptions such a continuation is unique and can therefore be approximated numerically. The stability properties are quantified through so called conditional stability estimates giving H\"older type stability in subsets of the bulk away from the domain boundary and logarithmic stability for bounds over the whole domain \cite{alessandrini2009stability}.

The pioneering method to solve the unique continuation problem was the quasi-reversibility method, in particular the mixed formulation introduced by Bourgeois \cite{bourgeois2005mixed}. In this case stability was ensured through Tikhonov regularization of the solution to the pde, leading to a perturbation error on the continuous level. For the discretization of such a formulation the stability of the continuous and discrete problems are handled by the same mechanism making it difficult to obtain error estimates that are sharp with respect to the approximation properties of the space and the stability of the target quantity. An alternative approach using a primal-dual stabilized approach was introduced by Burman \cites{Bur13,Bur14}. Here stabilizing terms on the discrete level ensured sufficient stability and for the first time error bounds were derived using the conditional stability of the unique continuation problem. This technique was further developed in \cite{Bur16, burman2018solving,burman2019unique,burman2020stabilized} to include high order polynomial approximation and more complex physics. When the polynomial order is increased some Tikhonov regularization needs to be introduced, but in this case the regularization parameter can be matched to the mesh-size and data perturbations and solution regularity in an optimal way \cite{BL24}. Some works have also been proposed in this context using non-conforming approximation \cite{WW20, Wang20, BDE21}. Only in \cite{BDE21} were error estimates derived using conditional stability. Another approach for unique continuation problems is to minimize the pde residual in dual norm
\cites{CKM22, DMS23}. In \cite{DMS23} a complete numerical analysis using conditional stability was proposed in this context. The relation between the quasi-reversibility method, Tikhonov regularization and least squares methods was further explored in \cite{bourgeois2018mixed}. Finally it was shown in \cite{BNO24} that error estimates reflecting the conditional stability and the best convergence of the residual allowed by the approximation space indeed are optimal for unique continuation. The sharpness of the conditional stability prohibits improvements on the feasible convergence rate. This result was sharpened by Monsuur and Stevenson in \cite{monsuur2024ultra}, where they showed that estimates in weak norm can be improved by minimizing the residual in the weakest possible topology.

The objective of the present work is to explore the interface between mixed quasi-reversibility methods, primal-dual-stabilized methods and residual minimization in dual norm. Our main result is a mixed quasi-reversibility method where the finite element spaces are chosen so that no stabilization of the forward problem is necessary, but optimal error estimates may nevertheless be obtained thanks to the inf-sup stability of the residual norms of the formulation. This shows the strong connection between the methods three methods. Their shared purpose is to control and minimize the bounding quantity of the conditional stability estimate. The different forms of the formulations are a reflection of the need for stability on both the discrete and the continuous level. As we see below, such stability can be obtained both through stabilization, allowing for the minimal dual space, or by choosing the test space carefully, allowing for a richer dual space, but avoiding discrete stabilization of the forward problem completely.

\section{Model problems and stability of the unique continuation}
In this section we will introduce the two model problems that we will discuss in this work and the associated conditional stability estimates.
\subsection{Model problem: unique continuation from data in the bulk}
Let $\Omega \in \mathbf{R}^d$, $d = 2,3$ be a polyhedral (polygonal), $\omega\subset \Omega $ be a subdomain that does not touch $\partial\Omega$. Let $(\cdot,\cdot)_U$, $\langle \cdot,\cdot\rangle_U$ and $<\cdot,\cdot>_U$ denote $L^2$, $H^1$ inner products and $H^{-1}-H^1$ pairing for (vector-valued) functions in $U$, respectively, and $||\cdot||_U$ denotes the $L^2$-norm in $U$. The \textbf{unique continuation problem} for Poisson's equation reads:
\begin{equation}\label{ucp}
    \left\{\begin{array}{rll}
        \Delta u &= f & \text{in $\Omega$} \\
         u&=q & \text{in $\omega$}
    \end{array}\right.,
\end{equation}
where $f\in H^{-1}(\Omega)$ and $q\in H^1(\omega)$. It is well known that the problem (\ref{ucp}) is ill-posed, by considering the classical example of Hardamard, e.g. \cite{hadamard2003lectures}. Such ill-posedness is that a solution may not exist and the bound in the following sense is absent:
\begin{equation*}
    ||u||_{H^1(\Omega)}\le C(||f||_{H^{-1}(\Omega)} + ||q||_{H^1(\omega)}).
\end{equation*}
This also causes the problem that the classical finite element discretization to (\ref{ucp}) may not be a stable and convergent scheme. Thus we define a continuous quasi-reversibility formulation for (\ref{ucp}) as follows: find $(u_\epsilon, \lambda_\epsilon) \in H^1(\Omega) \times H^1_0(\Omega)$, such that 
\begin{equation}\label{sys}
    \left\{\begin{array}{rll}
        \epsilon\langle u_\epsilon,v\rangle_\Omega+a(\lambda_\epsilon,v)+ (u_\epsilon,v)_\omega &= (q,v)_\omega & \forall v \in H^1(\Omega) \\
         a(u_\epsilon,w)-\langle \lambda_\epsilon,w\rangle_\Omega&=<f,w>_\Omega & \forall w\in H^1_0(\Omega)
    \end{array}\right.,
\end{equation}\\
for $\epsilon >0$. (\ref{sys}) is guaranteed to have a solution according to the Babuska-Lax-Milgram lemma \cite{Babuška1971}. It is not necessary for the solution to (\ref{sys}) to exactly fit the interior datum $q$ in $\omega$ or the weak formulation of the PDE $a(u_\epsilon,v)=<f,v>_\Omega$. Instead, we have the solution $u_\epsilon$ satisfying $\lim_{\epsilon\to 0}||u_\epsilon-u||_{H^1(\Omega)}=0$ and $\lambda$, as a Lagrange multiplier, satisfying $\lim_{\epsilon\to 0}||\lambda_\epsilon||_{H^1(\Omega)} = 0$. Following \cite{bourgeois2005mixed}, this can be seen by testing system (\ref{sys}) with $u_\epsilon-u$. This gives
\begin{equation}\label{u_epsilonbound}
    \epsilon\langle u_\epsilon,u_\epsilon\rangle_\Omega + (u_\epsilon-u,u_\epsilon-u)_\omega +\langle\lambda_\epsilon,\lambda_\epsilon\rangle_\Omega = \epsilon\langle u_\epsilon, u\rangle_\Omega.
\end{equation}
Since the latter two terms on the left-hand side are nonnegative, we have $||u_\epsilon||_{H^1(\Omega)}$ bounded by $ ||u||_{H^1(\Omega)}$ and thus $\{u_\epsilon\}$ has a weakly convergent subsequence converging to some function $\hat{u}\in H^1(\Omega)$. The boundedness of $u_\epsilon$ and (\ref{u_epsilonbound}) also gives 
\begin{equation}\label{weakbound}
\epsilon^\frac{1}{2}||u_\epsilon||_{H^1(\Omega)}+||u_\epsilon-u||_\omega+||\lambda_\epsilon||_{H^1(\Omega)}\le \epsilon^\frac{1}{2}||u||_{H^1(\Omega)}.
\end{equation}
Consequently by (\ref{weakbound}) we have $\hat{u}$ solves (\ref{ucp}), and by uniqueness of the solution to problem (\ref{ucp}), we conclude $\hat{u}=u$. Finally we have $||u_\epsilon-u||_{H^1(\Omega)}\to 0$, since
\begin{equation*}
    \langle u_\epsilon-u,u_\epsilon-u\rangle_\Omega = \langle u_\epsilon-u,u_\epsilon\rangle_\Omega-\langle u_\epsilon-u,u\rangle_\Omega\to 0.
\end{equation*}
This is due to the fact that the first term on the right hand side is negative by (\ref{u_epsilonbound}) and $u_\epsilon$ converges weakly to $u$.\\

\subsection{Model problem: Cauchy problem} The Cauchy problem we are interested in has the form: for $\Omega \in \mathbf{R}^d$, $d = 2,3$ a polyhedral (polygonal) domain and $\Gamma_0$ an open portion of $\partial\Omega$, find a function $u\in H^1(\Omega)$, such that
\begin{equation}\label{cp}
    \left\{\begin{array}{rlc}
         -\Delta u& =f & \texttt{in $\Omega$} \\
         u &=0 & \texttt{on $\Gamma_0$}\\
         \nabla u\cdot \nu &= \phi & \texttt{on $\Gamma_0$}
    \end{array},
    \right.
\end{equation}
where $\nu$ is the outward normal on $\Gamma_0$, $f\in L^2(\Omega)$, and $\phi\in H^{-\frac{1}{2}}(\Gamma_0)$. We introduce $H^{-\frac{1}{2}}(\Gamma_0)$ as the dual space of $H_{00}^\frac{1}{2}(\Gamma_0)$, which is defined as
\begin{equation*}
    H^\frac{1}{2}_{00}(\Gamma_0)=\{v\in L^2(\Gamma_0): \exists w\in H^1(\Omega), w|_{\Gamma_0}=v, w|_{\partial\Omega\backslash\Gamma_0}=0\}.
\end{equation*}
We denote $\Gamma_1 = \partial\Omega \backslash \Gamma_0$ the complement of $\Gamma_0$ and introduce the spaces 
\begin{equation}\label{v0v1}
\begin{array}{cc}
V_0=\{v\in H^1(\Omega): v|_{\Gamma_0} = 0\},& V_1 = \{v\in H^1(\Omega):w|_{\Gamma_1} = 0\}.
\end{array}
\end{equation}
Unlike the 'mixed' nature of the unique continuation problem (\ref{ucp}), the Cauchy problem has a one line weak form: find $u\in V_0$, such that
\begin{equation}\label{weakformcp}
\begin{array}{cc}
      a(u,v) = l(v) & \forall v\in V_1, \\
\end{array}
\end{equation}
where 
\begin{equation}\label{linearcp}
    l(v) = \int_\Omega fvdx + \int_{\Gamma_0}\phi vdS.
\end{equation}
Note that for simplicity we assume in (\ref{linearcp}) that $f\in L^2(\Omega)$ and $\phi\in H^\frac{1}{2}(\Gamma_0)$, but we naturally extend the definition of $l(v)$ to $V_{1}^{'}$, which represents the dual of $V_1$. We give a mixed formulation of quasi-reversibility for the Cauchy problem following the framework developed by Bourgeois et al. \cite{bourgeois2005mixed}, and \cite{bourgeois2018mixed} for an abstract formulation, which reads:
find $(u_\epsilon,\lambda_\epsilon)\in V_0\times V_1$, such that
\begin{equation}\label{sys2}
    \left\{\begin{array}{rll}
        \epsilon\langle u_\epsilon,v\rangle_\Omega+a(\lambda_\epsilon,v) &= 0 & \forall v \in V_0 \\
         a(u_\epsilon,w)-\langle \lambda_\epsilon,w\rangle_\Omega&=l(w) & \forall w\in V_1
    \end{array}\right..
\end{equation}\\
Following the same argument as the last part for unique continuation problems, the solution $(u_\epsilon,\lambda_\epsilon)$ to the system (\ref{sys2}) converges to $(u,0)$ in $H^1$ norm.
\begin{remark}
    For simplicity, we only discuss the stability and convergence of the numerical scheme of (\ref{cp}). The Cauchy problem with Non-homogeneous Dirichlet conditions $u|_{\Gamma_0}=g\in H^\frac{1}{2}(\Gamma_0)$ can also be treated in the same way by invoking a right inverse of $g$ (see, for example \cite[Section 33.1]{ern2021finite2}).
\end{remark}
\subsection{Continuous stability}
The above argument is standard for derive convergence for continuous quasi-reversibility formulation. See, for example \cite[Theorem 4]{bourgeois2005mixed}. However, this result does not provide information on the convergence rate with respect to $\epsilon$. To establish such a result, we will need conditional stability estimates or quantitative unique continuation. Let $s(u): H^1(\Omega) \to \mathbf{R}^+$ be a semi-norm and $d(u):H^1(\Omega) \to \mathbf{R}^+$ be the measurement of data in a PDE problem. We consider $\Theta(x): \mathbf{R}^+  \to \mathbf{R}^+$ as a function which is continuous and monotone increasing with $\lim_{x\to 0}\Theta(x)$, and then the conditional stability estimate takes the following form:\\

If $||u||_{H^1(\Omega)}\le E$ and $d(u)\le \eta$, then $s(u)\le \Theta(\eta).$\\

First, we state a more general problem to the Cauchy problem (\ref{cp}) by generalize the linear form (\ref{linearcp}). This formulation incorporates discrete functions in its solution space (c.f. \cite[Problem 1.6]{alessandrini2009stability}).
\begin{equation}\label{generalcp}
    a(u,v) = l(v):= \int_\Omega fv +F\cdot \nabla vdx+ \int_{\Gamma_0}\phi vdS, \forall v\in V_1,
\end{equation}
where $F\in L^2(\Omega,\mathbf{R}^d)$. Note that (\ref{linearcp}) is a special case when $F=0$.
Now we give two version of such conditional stability estimates, with respect to the unique continuation problem and the Cauchy problem. For the full details of such theorems, we refer the readers to \cite{alessandrini2009stability}. 
\begin{lemma}[Local conditional stability estimate]\thlabel{tbi}
    Assume $u\in H^1(\Omega)$ solves (\ref{ucp}) or (\ref{cp}) and satisfies that 
    \begin{equation*}
        ||u||_{H^1(\Omega)}\le E.
    \end{equation*}
    Define the data measurement $d(u)$ for the unique continuation problem as:
    \begin{equation}\label{datauc}
        d(u) = ||q||_\omega + ||f||_{H^{-1}(\Omega)}\le\eta
    \end{equation}
    and for the Cauchy problem as
    \begin{equation}\label{datacp}
        d(u) = ||\phi||_{H^{-\frac{1}{2}}(\Gamma_0)}+||f||_\Omega+||F||_\Omega\le \eta.
    \end{equation}
    Let $G\subset\Omega$ such that $\omega\subset G$ and $dist(G,\partial\Omega)>0$ for the unique continuation problem, or $dist(G, \Gamma_1)>0$ for the Cauchy problem. Then, there exists a constant $C>0$ and $\tau\in (0,1)$ depending on the geometries, such that
    \begin{equation*}
        s(u) := ||u||_{H^1(G)} \le C\eta^\tau (E+\eta)^{1-\tau}.
    \end{equation*}
\end{lemma}

\begin{lemma}\thlabel{global}
     Assume $u\in H^1(\Omega)$ solves (\ref{ucp}) or (\ref{cp}) and satisfies that 
    \begin{equation*}
        ||u||_{H^1(\Omega)}\le E.
    \end{equation*}
    Define the data measurement $d(u)$ for the unique continuation problem as:
    \begin{equation*}
        d(u) = ||q||_\omega + ||f||_{H^{-1}(\Omega)}\le\eta
    \end{equation*}
    and for the Cauchy problem as
    \begin{equation*}
        d(u) = ||\phi||_{H^{-\frac{1}{2}}(\Gamma_0)}+||f||_\Omega+||F||_\Omega\le \eta.
    \end{equation*}
    Then there exists a constant $C>0$ and $\tau\in(0,1)$ depending only on the geometry of $\Omega$, such that
    \begin{equation*}
        s(u) := ||u||_\Omega \le C(E+\eta)w(\frac{\eta}{E+\eta}),
    \end{equation*}
    where $w(t): (0,1)\to \mathbf{R}$ satisfies
    \begin{equation*}
        w(t) = \frac{1}{\log(t^{-1})^\tau}
    \end{equation*}
\end{lemma}
\begin{remark}
    In \thref{tbi}, we derived an error estimate for $||u||_{H^1(G)}$, while the reference theorems \cite[Theorem 1.7, Theorem 4.4]{alessandrini2009stability} only discussed the $L^2$-norm. The same process in the reference theorems work for $H^1$-norm as well, by using the $H^1$ version of Three sphere inequality. See, for example \cite[Corollary 3]{burman2019unique}.
\end{remark}
With \thref{tbi} and \thref{global} we can get a convergence rate for systems (\ref{sys}) and (\ref{sys2}). For example, let $e_\epsilon = u_\epsilon - u$, then by the second equation of (\ref{sys}) and (\ref{weakbound}), we have $||e_\epsilon||_\omega+||\Delta e_\epsilon||_{H^{-1}(\Omega)}\le \epsilon^\frac{1}{2}||u||_{H^1(\Omega)}$ and thus $||e_\epsilon||_{H^1(G)}\le C\epsilon^\frac{\tau}{2}||u||_{H^1(\Omega)}$. When it comes to discretization, the approximation error, when the measurement of data is exact, is expected to satisfy 
\begin{equation}\label{optimalrate}
    ||u_h-u||_{H^1(G)}\le Ch^\tau||u||_{H^2(\Omega)}
\end{equation}
with mesh size $h$ with a proper choice of the regularization parameter $\epsilon$ (in this case, $\epsilon \simeq h^2$). This convergence rate has recently been shown to be optimal in \cite{BNO24}. However, a naive $H^1$-conforming discretization for (\ref{sys}) does not achieve such a convergence rate; see \cite[Section 4]{burman2018solving}. This is because the $H^1$-regularization $\epsilon\langle\cdot,\cdot\rangle_\Omega$ alone can not both serve as Tikhonov regularization and numerical stabilizer that generates optimal convergence rate for the discrete system. Recently, a \textbf{ discretize then regularize} method has been developed to tackle this problem. See \cite{burman2018solving,burman2019unique,burman2020stabilized,BL24}, where stabilizers at the have been introduced at the discrete level to generate an optimal convergence.  The regularization parameter $\epsilon$ is  be mesh dependent and tuned against the conditional stability and the numerical stabilization.

In this paper, we propose a new finite element method based on the quasi-reversibility method, instead of introducing stabilizing terms to ensure numerical stability, we choose the trial and test spaces carefully in the discretization of (\ref{sys}) and (\ref{sys2}). The objective is to achieve an inf-sup stable formulation and through weak consistency the optimal convergence rate suggested by (\ref{optimalrate}).

\section{The finite element formulation of quasi-reversibility}
In this section we introduce our finite element formulation of quasi-reversibility. 
\subsection{Finite element preliminaries}\label{prelim}
Let $\{\mathcal{T}_h\}_h$ be a family of shape-regular and quasi-uniform tesselations of $\Omega$ in nonoverlapping simplices such that $\forall K,K'\in \mathcal{T}_h$, $K\cap K'$ is either empty, a common vertex, or a common facet (edge). Let the diameter of a simplex $K$ be $h_K$ and the outward normal be $n_K$. The global mesh parameter is defined by $h:\max_{K\in\mathcal{T}}h_K$. Let $\mathcal{F}$ denote the collection of all faces ($d=3$) or edges ($d=2$) with respect to the mesh $\mathcal{T}_h$.
\begin{equation*}
    \begin{array}{cc}
       \mathcal{F}_I := \{F\in\mathcal{F}:F\not\subset \partial\Omega\},  & \mathcal{F}_l = \{F\in\mathcal{F}: F^o\cap \Gamma_l \neq \emptyset\},l\in\{0,1\},
    \end{array}
\end{equation*}
where $\Gamma_0$ is the part of the boundary where the solution is homogeneous and Neumann data is given. We assume $F_{\Gamma_0}\cap F_{\Gamma_1} = \emptyset$. Additionally, $\Gamma_0 =\emptyset$ in the unique continuation problem (\ref{ucp}). Let $n_F$ be a normal on $F$, whose orientation is arbitrary but fixed. Let $k\in\mathbb{N}$, $D$ be a domain in $\mathbb{R}^d$, $d=1,2,3$, and $\mathbf{P}_k(D)$ be the space of polynomials with degree $\le k$ on $D$. We introduce the space of all $L^2$ functions which are polynomial on each element $K$
\begin{equation*}
    X^k_h := \{x_h\in L^2(\Omega), x_h|_K \in \mathbf{P}_k(K), \forall K\in \mathcal{T}_h\}.
\end{equation*}
We also define its $H^1$-conforming subspace and that with homogeneous boundary
\begin{equation}\label{H1spaces}
\begin{array}{cc}
    V^k_h:= H^1(\Omega)\cap X^k_h, & W^k_h:= H_0^1(\Omega)\cap X^k_h\\
    V^k_{h0}: V_0\cap X^k_h, & V^k_{h1}: V_1\cap X^k_h,
\end{array}
\end{equation}
where $V_0,V_1$ are defined in (\ref{v0v1}). Now we mention some notations we will be using repeatedly. Let $\mathcal{L}$ be a differential operator (such as $\Delta,\nabla$) then $\mathcal{L}_h$ ($\Delta_h,\nabla_h$) represents its discrete version, i.e.
\begin{equation*}
\begin{array}{cc}
     \mathcal{L}_hx_h = \sum_{K\in\mathcal{T}_h}T(\mathcal{L}x_h|_K),& \forall x_h\in X^k_h \\
\end{array}
\end{equation*}
where $T$ is the extension operator to extend the domain to $\Omega$ with value $0$. For any $x_h\in X_h$, we denote its average function and the jump function of its normal gradient across $F\in \mathcal{F}_I$ by 
\begin{equation*}
\begin{aligned}
 \{x_h\} &:= \frac{1}{2}(x_h|_{K_1}+x_h|_{K_2}),\\
\llbracket\nabla x_h\cdot n\rrbracket &:=\nabla x_h|_{K_1}\cdot n_{K_1} + \nabla x_h|_{K_2}\cdot n_{K_2}.
\end{aligned}
\end{equation*}
Moreover, $\forall x_h, y_h \in X_h$, define 
\begin{equation*}
    \mathcal{J}_h(x_h,y_h) = \sum_{F\in \mathcal{F}_I}\int_Fh\llbracket\nabla x_h\cdot n\rrbracket\cdot\llbracket\nabla y_h\cdot n\rrbracket dS.
\end{equation*}
We present here also the well-known trace and inverse inequalities (see, for example, \cite[Section 1.4.3]{di2011mathematical}). We have $\forall K\in \mathcal{T}_h$:

(1) Continuous trace inequality.
\begin{equation}\label{tc1}
    ||v||_{\partial K}\le C(h^{-\frac{1}{2}}||v||_U+h^{\frac{1}{2}}||\nabla v||_K), \;\; \forall v\in H^1(K).
\end{equation}

(2) Discrete trace inequality.
\begin{equation}\label{tc2}
    ||\nabla v_h \cdot n||_{\partial K}\le Ch^{-\frac{1}{2}}||\nabla v_h||_K, \;\; \forall v_h\in \mathbf{P}_k(K).
\end{equation}

(3) Inverse inequality.
\begin{equation}\label{inv}
    ||\nabla v_h||_K\le Ch^{-1}||v_h||_K, \;\; \forall v_h\in\mathbf{P}_k(K).
\end{equation}\\
We also include the interpolation operators that we will use repeatedly.

(1) The Scott-Zhang operator $\pi^k_{sz}: H^{\frac{1+\epsilon}{2}}(\Omega) \to X_h^k$. It maintains a homogeneous boundary condition and has the optimal approximation property, see \cite{ciarlet2013analysis}:
$$ ||v-\pi^k_{sz}v||_{H^m(\mathcal{T}_h)} \le Ch^{s-m}|v|_{H^{s}(\Omega)}, $$
whenever $v\in H^{s}(\Omega)$ with $\frac{1}{2}<s\le k+1$ and $m\in \{0:[s]\}$.\\

(2) $L^2$-projection $\pi^k_h: L^2 \to X_h^k$ with $k\ge 1$. It is the best approximation in the $L^2$-norm, and thus it has optimal approximation property in $L^2$-norm
\begin{equation*}
    ||v-\pi^k_hv||_\Omega \le C h^s|v|_{H^{s}(\Omega)},
\end{equation*}
whenever $v\in H^s(\Omega)$ with $0\le s \le k+1$.\\

\subsection{Discrete formulation: unique continuation}
 We now specify our discrete formulation of quasi-reversibility for unique continuation problem, which reads: find $(u_h,\lambda_h)\in V_h\times W_h$, such that
\begin{equation}\label{dissys}
    \left\{\begin{array}{rll}
        \epsilon\langle u_h,v_h\rangle_\Omega+a_h(\lambda_h,v_h)+ (u_h,v_h)_\omega &= (q,v_h)_\omega & \forall v_h \in V_h \\
         a_h(u_h,w_h)-\gamma^2\langle \lambda_h,w_h\rangle_{\mathcal{T}_h}&=(f_h,w_h)_\Omega & \forall w_h\in W_h
    \end{array}\right.,
\end{equation}
where $0<\gamma<1$, $V_h = V_h^k$ , and $W_h^1\subset W_h\subset X_h^m$ for $k,m\in \mathbf{N}$. $a_h(x,y) = (\nabla_hx,\nabla_hy)_{\mathcal{T}_h}$, and $\langle\cdot,\cdot\rangle_{\mathcal{T}_h}$ represents the discrete $H^1$ inner product. $f_h$ is the interpolant of $f$ in $W_h^m$ such that $\forall w_h\in W_h^m$,
\begin{equation}\label{f_h}
<f,w_h> = (f_h,w_h)_\Omega.
\end{equation}
Now we test (\ref{f_h}) by $w_h=f_h$. By using an inverse inequality we obtain
\begin{equation}
    ||f_h||_\Omega \le ||hf||_{H^{-1}(\Omega)}.
\end{equation}

\begin{remark}\thlabel{rm1}
    As shown in \cite[Section 4]{burman2018solving}, by taking $W_h=V_h\cap H^1_0(\Omega)$ the scheme fail to generate an optimal convergence solution even by choosing an optimal regularization parameter $\epsilon$ with respect to $h$. Moreover, the choice of the optimal $\epsilon$ depends on the generally unknown conditional stability parameter $\tau$.
\end{remark}
\begin{remark}
    Note that in our framework conforming $V_h$ is always used but $W_h$ may be non-conforming, and thus we need to use the broken bi-linear forms in system (\ref{dissys}).
\end{remark}
Our aim is to show that scheme (\ref{dissys}) converges at an optimal rate with appropriate choice of the finite dimensional space $W_h$. To achieve this, we need to bound the discretization error and the interpolation error under a 'correct' norm. Inspired by the conditional stability estimate, we define a norm as the data term in {\ref{datauc})}:
\begin{equation}\label{ucnorm}
    ||u||_{uc}:=d(u) =||u||_\omega+||\Delta u||_{H^{-1}(\Omega)}, \forall u\in H^1(\Omega)
\end{equation}
Note that $||u||_{uc}$ is indeed a norm by using \thref{global}: $||u||_{uc}=0\to u=0$. This norm makes sense at the discrete level since we restrict ourselves to conforming subspaces for $u_h$. If it can be shown that $||u_h-u||_{uc}$ converges optimally, then together with \thref{tbi} or \thref{global} we can derive a convergence rate under the $H^1$ norm. Now we decompose $||u_h-u||_{uc}$ by observing
\begin{equation*}
    ||u_h-u||_{uc} \le ||u_h-\pi_{sz}^ku||_{uc}+||\pi_{sz}^ku-u||_{uc},
\end{equation*}
where $\pi_{sz}^k$ is the Scott-Zhang operator. First by the definition of the norm (\ref{ucnorm}), we have the bound
\begin{equation*}
||v||_{uc}\lesssim ||v||_{H^1(\Omega)}, \forall v\in H^1(\Omega).
\end{equation*}
This immediately shows that the interpolation error $||\pi_{sz}^ku-u||_{uc}$ is vanishing optimally. Now the only remaining part is to bound the discrete error $e_h = u_h-\pi^k_{sz}u$. For $||e_h||_{H^{-1}(\Omega)}$ we have: $\forall w\in H_0^1(\Omega)$
\begin{equation}\label{rbounduc}
\begin{aligned}
    a(e_h,w)  =& a(e_h,w-\pi^1_{sz}w)+a(u_h-u,\pi^1_{sz}w)+a(u-\pi_{sz}^ku,\pi^1_{sz}w)\\
    =&  \sum_{F\in \mathcal{F}_I}\int_F\llbracket\nabla e_h\cdot n\rrbracket(w-\pi^1_{sz}w)dS + ||u-\pi^k_{sz}u||_{H^1(\Omega)}||w||_{H^1(\Omega)}\\
    &+(\Delta_h e_h, w-\pi^1_{sz}w)_\Omega+\gamma^2\langle\lambda_h,\pi^1_{sz}w\rangle_{\mathcal{T}_h}\\
    \le &C\left(||h\Delta_h e_h||_\Omega+\mathcal{J}_h(e_h,e_h)^\frac{1}{2}\right. \\
    & \left.+||\gamma\lambda_h||_{H^1(\mathcal{T}_h)}+||u-\pi^k_{sz}u||_{H^1(\Omega)}\right)||w||_{H^1(\Omega)},
    \end{aligned}
\end{equation}
where the second equality applies the fact that $W_h^1\subset W_h$ and the inequality is a direct application of the trace inequality and the Cauchy-Schwartz inequality. From (\ref{rbounduc}) we observe that if we control the convergence rate for the right-hand side, together with $||e_h||_\omega$, the scheme will generate an optimal convergence. For all $(v_h,w_h)\in V_h\times W_h$, define:
\begin{equation}\label{triplenormuc}
    |||(v_h,w_h)|||^2_{uc}:= \epsilon||v_h||^2_{H^1(\Omega)}+||v_h||^2_\omega + \mathcal{J}_h(v_h,v_h)+||h\Delta_h v_h||^2_\Omega+||\gamma w_h||^2_{H^1(\mathcal{T}_h)},
\end{equation}
then we have
\begin{equation*}
    ||e_h||_{uc} \lesssim |||(e_h,\lambda_h)|||_{uc}.
\end{equation*}
Therefore our goal now is to show $|||(e_h,\lambda_h)|||_{uc}$ converges optimally. 
\begin{remark}
    The reason why we add $\epsilon||v_h||^2_{H^1(\Omega)}$ in (\ref{triplenormuc}) is the remaining terms, even though they define a norm in the discrete space, do not control the $H^1$ norm of the discrete solution uniformly with mesh refinement.
\end{remark}
\begin{remark}
    It is worth noting that, since $e_h$ is defined in $V_h$, we do not need to assume higher regularity for the continuous solution $u$ but only $u\in H^1(\Omega)$ (see \cite{BL24}).
\end{remark}

\subsection{Discrete formulation: Cauchy problem}Using the same notations we also propose the discrete formulation for the Cauchy problem: find $(u_h,\lambda_h)\in V_{0h}\times V_{1h}$, such that
\begin{equation}\label{dissys2}
    \left\{\begin{array}{rll}
        \epsilon\langle u_h,v_h\rangle_\Omega+a_h(\lambda_h,v_h) &= 0 & \forall v_h \in V_{0h} \\
         a_h(u_h,w_h)-\gamma^2\langle \lambda_h,w_h\rangle_\Omega&=(f,w_h)_\Omega+(\phi_h,w_h)_{\Gamma_0} & \forall w_h\in V_{1h}
    \end{array}\right.,
\end{equation}
where $V_{0h} = V_{0h}^k$ and $V_{1h}^1\subset V_{1h}\subset X_h^m$ for $k,m\in \mathbf{N}$. $\phi_h$ is the interpolant of $\phi$, defined by: $\forall w_h\in V_{1h}^m$,
\begin{equation}\label{phi_h}
    <\phi,w_h>_{H^{-\frac{1}{2}}(\Gamma_0)\to H^\frac{1}{2}(\Gamma_0)} = (\phi_h,w_h)_{\Gamma_0}.
\end{equation}
Note that the above interpolation does not guarantee a unique interpolant $\phi_h$. We choose the interpolant having only non zero degrees of freedoms on the facets $F\in \mathcal{F}$ such that $F\subset \Gamma_0$. Then by choosing $w_h$ such that $w_h|_{\Gamma_0}=\phi_h$, applying inverse inequality (\ref{inv}) and inverse trace inequality (\ref{invtrace}). We have the following estimate
\begin{equation}
    ||\phi_h||_{\Gamma_0}^2 = (\phi_h,w_h)_{\Gamma_0} \le ||\phi||_{H^{-\frac{1}{2}}(\Gamma_0)}||w_h||_{H^{\frac{1}{2}}(\Gamma_0)}.
\end{equation}
Then proceeding with a trace inequality, an inverse inequality and an inverse trace inequality
\[
||w_h||_{H^{\frac{1}{2}}(\Gamma_0)} \le C |w_h|_{H^1(\Omega)}\le C h^{-1} ||w_h||_{L^2(\Omega)} \le C h^{-\frac12} ||\phi_h||_{\Gamma_0}.
\]
Now we define the norm for functions in $V_0$ we want to control as in the discussion of the unique continuation problem. We claim that to bound $d(u-u_h)$ with $d(\cdot)$ defined in (\ref{datacp}) it is equivalent to define :
\begin{equation*}
    ||v||_{cp} = \sup_{w\in V_1}\frac{a(v,w)}{||w||_{H^1(\Omega)}}.
\end{equation*}
We give the details of this claim in Appendix \ref{equinorms}. In this formulation we naturally have
\begin{equation*}
    ||v||_{cp} \le C||v||_{H^1(\Omega)},
\end{equation*}
which gives immediately that $||u-\pi_{sz}^ku||_{cp}$ vanishes optimally. To consider the discretization error $||u_h-\pi_{sz}^ku||_{cp}$, we use the same decomposition $a(e_h,w) = a(e_h,w-\pi_{sz}^1w)+a(u_h-u,\pi_{sz}^1w)+a(u-\pi_{sz}^ku,\pi_{sz}^1w)$. Then by the same technique as (\ref{rbounduc}), we obtain: $\forall w\in V_1$
\begin{equation}\label{rboundcp}
\begin{aligned}
    a(e_h,w)\le &C\left(||h\Delta_h e_h||_\Omega+\mathcal{J}_h(e_h,e_h)^\frac{1}{2} + ||h^\frac{1}{2}(\nabla e_h\cdot\nu)||_{\Gamma_0}\right. \\
    &\left.+||\gamma\lambda_h||_{H^1(\mathcal{T}_h)}+||u-\pi^k_{sz}u||_{H^1(\Omega)}\right)||w||_{H^1(\Omega)}
    \end{aligned}.
\end{equation}
Thus we define: $\forall (v_h,w_h)\in V_{0h}\times V_{1h}$
\begin{equation}\label{triolenormcp}
\begin{aligned}
    |||(v_h,w_h)|||^2_{cp}:= & \epsilon||v_h||^2_{H^1(\Omega)}+\mathcal{J}_h(v_h,v_h)+h||(\nabla v_h\cdot\nu)||^2_{\Gamma_0}\\
    &+||h\Delta_h v_h||^2_\Omega+||\gamma w_h||^2_{H^1(\mathcal{T}_h)}.
\end{aligned}
\end{equation}
Consequently we have
\begin{equation*}
    ||e_h||_{cp} \le C|||(e_h,\lambda_h)|||_{cp}.
\end{equation*}
We conclude in this section that once we can show that if $|||(e_h,\lambda_h)|||_{uc}$ and $|||(e_h,\lambda_h)|||_{cp}$ vanishes optimally in the following sense
\begin{equation}\label{triplenormconvergenceuc}
    |||(e_h,\lambda_h)|||_{uc} \le Ch^{s-1}||u||_{H^{s}(\Omega)},
\end{equation}
and 
\begin{equation}\label{triplenormconvergencecp}
    |||(e_h,\lambda_h)|||_{cp}\le Ch^{s-1}||u||_{H^{s}(\Omega)},
\end{equation}
then the discrete solutions to systems (\ref{dissys}) and (\ref{dissys2}) achieve an optimal convergence rate suggested by (\ref{optimalrate}).

\subsection{Stability and error analysis}\label{Stability and error analysis}
In this section, we discuss the discrete stability for systems (\ref{dissys}) and (\ref{dissys2}). Then we introduce some additional stability assumptions under which we would be able to justify (\ref{triplenormconvergenceuc}) and (\ref{triplenormconvergencecp}). In the next section, we will give explicit construction of finite element spaces for which these assumptions will be satisfied. 

First we introduce the compact-form operators:
\begin{equation}\label{A}
\begin{aligned}
    A_{uc}&:= \epsilon\langle u_h,v_h\rangle_\Omega+a_h(\lambda_h,v_h)+ (u_h,v_h)_\omega +a_h(u_h,w_h)-\langle \lambda_h,w_h\rangle_{\mathcal{T}_h} \\
    A_{cp}&:=\epsilon\langle u_h,v_h\rangle_\Omega+a_h(\lambda_h,v_h)+a_h(u_h,w_h)-\langle \lambda_h,w_h\rangle_{\mathcal{T}_h}.
\end{aligned}
\end{equation}
We have the following discrete stability for the system (\ref{dissys}) and (\ref{dissys2})
\begin{prop}\thlabel{conditionnumber}
    Let $0<h, \epsilon,\gamma<1$. The discrete systems (\ref{dissys}) and (\ref{dissys2}) have unique solutions $(u_h,\lambda_h)$, with a Euclidean condition number
    \begin{equation*}
        \mathcal{K}_2 \le \frac{C}{\epsilon h^2}.
    \end{equation*}
\end{prop}
\begin{proof}
    First we observe that $A_{uc}[(u_h,\lambda_h),(u_h,-\lambda_h)] = \epsilon\langle u_h,u_h\rangle_\Omega +(u_h,u_h)_\omega +||\gamma\lambda_h||^2_{H^1(\Omega)}$, where $(u_h,\lambda_h)$ is the Cartesian product. Thus we have
    \begin{equation}\label{psi_h}
        \Psi_h := \inf_{(u_h,\lambda_h)}\sup_{(v_h,w_h)}\frac{A_{uc}[(u_h,\lambda_h),(v_h,w_h)]}{||(u_h,\lambda_h)||_\Omega||(v_h,w_h)||_\Omega}\ge \epsilon.
    \end{equation}
    Then define 
    \begin{equation*}
        \Upsilon: =\sup_{(u_h,\lambda_h)}\sup_{(v_h,w_h)}\frac{A_{uc}[(u_h,\lambda_h),(v_h,w_h)]}{||(u_h,\lambda_h)||_\Omega||(v_h,w_h)||_\Omega}.
    \end{equation*}
    Applying the inverse inequality (\ref{inv}), we have the upper bound
    \begin{equation*}
    \begin{aligned}
        A_{uc}[(u_h,\lambda_h),(v_h,w_h)]\lesssim &\frac{\epsilon}{h^2}||u_h||_\Omega||v_h||_\Omega + \frac{1}{h^2}(||u_h||_\Omega||w_h||_\Omega+||\lambda_h||_\Omega||v||_\Omega)\\
        &+\frac{1}{h^2}||\lambda_h||_\Omega||w_h||_\Omega,
    \end{aligned}
         \end{equation*}
    this gives $\Upsilon \le \frac{C}{h^2}$. Then by \cite[Theorem 3.1]{ern2006evaluation}, we conclude
    \begin{equation*}
        \mathcal{K}_2 = \frac{\Upsilon}{\Psi} \le \frac{C}{\epsilon h^2}.
    \end{equation*}
    The proof for $A_{cp}$ is identical.
    
\end{proof}
\thref{conditionnumber} guarantees that the discrete schemes (\ref{dissys}), (\ref{dissys2}) attain solutions, without specifying the spaces where the Lagrange multiplier $\lambda_h$ lives. However, the choice of $W_h$ and $V_{1h}$ will strongly influence the stability property of the method, as discussed in \thref{rm1}. The idea to generate an optimal convergence rate, is to minimize the residual under a "correct" norm, which are the triple norms $|||(\cdot,\cdot)|||_{uc}$ and $|||(\cdot,\cdot)|||_{cp}$. Now we give the assumption on the strengthened stability under the triple norms.
\begin{ass}[Stability]\thlabel{infsup}
    Let the discrete operators be as defined in (\ref{A}). Then there is some positive constant $C$ independent of the mesh size $h$ and the regularization parameter $\epsilon$, such that $\forall (u_h,\lambda_h) \in V_h\times W_h$,
    \begin{equation*}
        \sup_{(v_h,w_h)\in V_h\times W_h}\frac{A_{uc}[(u_h,\lambda_h),(v_h,w_h)]}{|||(v_h,w_h)|||_{uc}} \ge C|||(u_h,\lambda_h)|||_{uc},
    \end{equation*}
    and $\forall (v_h,w_h)\in V_{0h}\times V_{1h}$
    \begin{equation*}
        \sup_{(v_h,w_h)\in V_{0h}\times V_{1h}}\frac{A_{cp}[(u_h,\lambda_h),(v_h,w_h)]}{|||(v_h,w_h)|||_{cp}} \ge C|||(u_h,\lambda_h)|||_{cp}.
    \end{equation*}
\end{ass}
This inf-sup stability shows that the method gives sufficient control of the residual quantities to achieve optimal convergence. We mention that the estimate (\ref{psi_h}) itself fails to derive convergence for $u_h-u$ (but for $u_h-u_\epsilon$) because the constant depends on the regularization parameter $\epsilon$. Apart from the stability of the discrete system, we also make the following consistency assumption.
\begin{ass}[Consistency]\thlabel{consistency}
    $W_h$ and $V_{1h}$ are chosen such that the weak consistency estimates hold, i.e., if $u\in H^s(\Omega)$ with $s\ge 1$ is the solution to (\ref{ucp}), then $\forall w_h \in W_h$
    \begin{equation*}
        |a_h(u,w_h) - (f_h,w_h)_\Omega|\le Ch^{s-1}||u||_{H^s(\Omega)}||w_h||_{H^1(\mathcal{T}_h)}.
    \end{equation*}
    Otherwise, if $u\in H^s(\Omega)$ is the solution to (\ref{cp}), then $\forall w_h \in V_{1h}$
    \begin{equation*}
        |a_h(u,w_h) - (f,w_h)_\Omega-(\phi_h,w_h)_{\Gamma_0}|\le Ch^{s-1}||u||_{H^s(\Omega)}||w_h||_{H^1(\mathcal{T}_h)}.
    \end{equation*}
\end{ass}
\begin{remark}
    \thref{consistency} is not naturally fulfilled unless $W_h$ ($V_{1h}$) are assumed to be $H^1_0(\Omega)$ ($V_1$) conforming, in which case the left hand side of the above equations are $0$ due to the definition of the interpolants $f_h$ and $\phi_h$ (c.f. (\ref{f_h}), (\ref{phi_h})).
\end{remark}
Now under \thref{consistency} and \thref{infsup}, we will be able to derive an optimal convergence estimate under the triple norms for (\ref{dissys}) and (\ref{dissys2}).

\begin{theorem}\thlabel{triplenormconv}
    Let $u\in H^s(\Omega)$ with $1\le s \le k+1$ be the solution to (\ref{ucp}) or (\ref{cp}) and $(u_h,\lambda_h)$ be the solution to the discrete system (\ref{dissys}) or (\ref{dissys2}), which satisfy \thref{consistency} and \thref{infsup}. Denote $e_h = u_h-\pi_{sz}^ku$ then there exists some constant $C>0$, such that
    \begin{equation*}
        |||(e_h,\lambda_h)|||_{uc} \le C(h^{s-1}+\epsilon^\frac{1}{2})||u||_{H^s(\Omega)},
    \end{equation*}
    and 
    \begin{equation*}
        |||(e_h,\lambda_h)|||_{cp} \le C(h^{s-1}+\epsilon^\frac{1}{2})||u||_{H^s(\Omega)}.
    \end{equation*}
\end{theorem}
\begin{proof}
    We will show the proof for the Cauchy problem, and that for the unique continuation problem can be derived by simply repeating the process.

    From \thref{infsup}, we only need to show that $\forall (v_h,w_h)\in V_{0h}\times V_{1h}$,
    \begin{equation*}
        A_{cp}[(e_h,\lambda_h), (v_h,w_h)] \le C(h^{s-1}+\epsilon^\frac{1}{2})||u||_{H^s(\Omega)}|||(v_h,w_h)|||_{cp}.
    \end{equation*}
    Recall the definition of $A_{cp}$ and plug in (\ref{dissys2}), we have
    \begin{equation*}
    \begin{aligned}
        A_{cp}[(e_h,\lambda_h), (v_h,w_h)] &=\epsilon\langle e_h,v_h\rangle_\Omega+a_h(\lambda_h,v_h)+a_h(e_h,w_h)-\langle \lambda_h,w_h\rangle_{\mathcal{T}_h}\\
        &=(f_h,w_h)_\Omega+(\phi_h,w_h)_{\Gamma_0}-a_h(\pi_{sz}^ku,w_h)-\epsilon\langle \pi_{sz}^ku,v_h\rangle_\Omega\\
        &\le |a_h(u-\pi_{sz}^ku,w_h)|+\epsilon|\langle \pi_{sz}^ku,v_h\rangle_\Omega|+ Ch^{s-1}||u||_{H^s(\Omega)}||w_h||_{H^1(\mathcal{T}_h)}\\
        &\le C(h^{s-1}+\epsilon^\frac{1}{2})||u||_{H^s(\Omega)}(||w_h||_{H^1(\mathcal{T}_h)}+||\epsilon^\frac{1}{2}v_h||_{H^1(\Omega)})\\
        &\le C(h^{s-1}+\epsilon^\frac{1}{2})||u||_{H^s(\Omega)}|||(v_h,w_h)|||_{cp},
        \end{aligned}
    \end{equation*}
    where the first inequality uses \thref{consistency} and the second the approximation properties of the Scott-Zhang interpolant.
    
\end{proof}
\begin{co}\thlabel{boundandconv}
    Let $u$ and $u_h$ be the same as in \thref{triplenormconv}. Set $h = \epsilon^{\frac{1}{2s-2}}$. Then we have the following boundedness and convergence estimates:
    
    (1) For the unique continuation problem,
    \begin{equation*}
        \begin{array}{cc}
          ||u_h-u||_{H^1(\Omega)}\le C||u||_{H^s(\Omega)},   & ||u_h-u||_{uc} \le C\epsilon^\frac{1}{2}||u||_{H^s(\Omega)}.
        \end{array}
    \end{equation*}
    
    (2) For the Cauchy problem,
    \begin{equation*}
    \begin{array}{cc}
        ||u_h-u||_{H^1(\Omega)}\le C||u||_{H^s(\Omega)},  & ||u_h-u||_{cp} \le C\epsilon^\frac{1}{2}||u||_{H^s(\Omega)}.
        \end{array}
    \end{equation*}
\end{co}
\begin{proof}
    For the boundedness, we use the fact that
    \begin{equation*}
        ||e_h||_{H^1(\Omega)}\le \epsilon^{-\frac{1}{2}}|||(e_h,\lambda_h)|||_{uc}\le C(1+\epsilon^{-\frac{1}{2}}h^{s-1})||u||_{H^s(\Omega},
    \end{equation*}
    and that the Scott-Zhange interpolant is bounded, and for the concergence, we use 
    \begin{equation*}
        ||e_h||_{uc} \le C|||(e_h,\lambda_h)|||_{uc},
    \end{equation*}
    and that $||u-\pi_{sz}^ku||_{uc}$ converges optimally.
    
\end{proof}

\begin{theorem}\thlabel{conv}
    Let $u$ and $u_h$ be defined as in \thref{triplenormconv}. Let $0<\epsilon <1$ and set $h = \epsilon^\frac{1}{2s-2}$. Then there exists some constant $C>0$ and $\tau\in (0,1)$, such that
    \begin{equation*}
        ||u-u_h||_{\Omega} \le \frac{C}{(-\log{\epsilon})^\tau}||u||_{H^s(\Omega)},
    \end{equation*}
    and for any $G$ satisfying the condition in \thref{tbi}, we have
    \begin{equation*}
        ||u-u_h||_{H^1(G)} \le C\epsilon^\frac{\tau}{2}||u||_{H^s(\Omega)}.
    \end{equation*}
    Moreover, if we convert the estimate in terms of discretization, we have
    \begin{equation*}
        ||u-u_h||_{\Omega} \le \frac{C}{(-(2s-2)\log{h})^\tau}||u||_{H^s(\Omega)},
    \end{equation*}
    and
    \begin{equation*}
        ||u-u_h||_{H^1(G)} \le Ch^{\tau(s-1)}||u||_{H^s(\Omega)}.
    \end{equation*}
\end{theorem}
\begin{proof}
    The estimates are obtained by simply plugging the bounds presented in \thref{boundandconv} into \thref{tbi} and \thref{global}.
    
\end{proof}
\begin{remark}
    In \thref{conv}, we derive two versions of optimal convergence. We mention that the $\epsilon$-version serves as the most coarse discretization by setting $h = \epsilon^\frac{1}{2(s-1)}$. This is possible, thanks to the inf-sup stability. The $h$-version provides the optimal convergence rate that can be achieved with mesh refinement before reaching the precision threshold determined by the regularization parameter $\epsilon$.
\end{remark}
\subsection{Optimal convergence under $L^2$-norm}
In \thref{conv}, we obtain an optimal convergence rate with respect to the $H^1$-norm in $G$. However, such a rate is not optimal for the $L^2$ error. In fact, an optimal convergence rate is not expected under the schemes (\ref{dissys}) and (\ref{dissys2}) even under \thref{consistency} and \thref{infsup} (See \cite{BL24}). To settle this problem, we need to employ another conditional stability estimate recently established in \cite[Section 2]{monsuur2024ultra}. Taking the unique continuation problem as an example, the estimate reads
\begin{equation}\label{ultraweak}
    ||u||_G \le C(||u||_\omega+||f||_{H^{-2}(\Omega)})^\tau||u||^{1-\tau}_\Omega.
\end{equation}
Now instead of applying the scheme (\ref{dissys}), we propose a stabilized scheme
\begin{equation}\label{dissysL2}
    \left\{\begin{array}{rll}
        \epsilon\langle u_h,v_h\rangle_\Omega+a_h(\lambda_h,v_h)+ h^{-2}(u_h,v_h)_\omega &= h^{-2}(q,v_h)_\omega & \forall v_h \in V_h \\
         a_h(u_h,w_h)-\gamma^2s^*(\lambda_h,w_h)&=(f_h,w_h)_\Omega & \forall w_h\in W_h
    \end{array}\right..
\end{equation}
We define $s^*$ as
\begin{equation*}
    s^*(\lambda_h,w_h)=\mathcal{J}_h(\lambda_h,w_h)+\int_{\partial \Omega}h\partial_n\lambda_h\partial_nw_hdS+( h\Delta_hz_h,h\Delta_hw_h),
\end{equation*}
where $\partial_n$ is the normal derivative. With this new scheme we would be able to obtain optimal convergence rate for the $L^2$-error.
\begin{co}\thlabel{L2conv}
    Let $u\in H^s(\Omega)$ with $s\ge 1$ be the solution to (\ref{ucp}) with the regularization parameter $\epsilon=h^{2s}$, and $u_h$ be the solution to (\ref{dissysL2}), then there exists some constant $C>0$, such that
    \begin{equation*}
        ||u_h-u||_G \le Ch^{s\tau}||u||_{H^s(\Omega)}.
    \end{equation*}
\end{co}
\begin{proof}
    The estimate can be obtained by combining (\ref{ultraweak}), \thref{triplenormconv} and \cite[Theorem 4.4]{BL24}.
    
\end{proof}
\subsection{Stability under perturbation}
In the previous discussion we assumed the measurement of data is exact (such as the bulk data $q$, the right hand side $f$ and the normal derivative $\phi$). However, the measured data can be subject to perturbations. In this section, we discuss the stability of the scheme under some certain level of perturbation and how such perturbations affect accuracy of the approximation.

The perturbation in data are assumed to be $L^2$ functions, and we denote the perturbed data by $\hat q,\hat f,\hat\phi$. Then the resulting systems (c.f. (\ref{dissys}),(\ref{dissys2})) with perturbation read:
\begin{equation}\label{disucppert}
    \left\{\begin{array}{rll}
        \epsilon\langle u_h,v_h\rangle_\Omega+a_h(\lambda_h,v_h)+ (u_h,v_h)_\omega &= (\hat q,v_h)_\omega & \forall v_h \in V_h \\
         a_h(u_h,w_h)-\gamma^2\langle \lambda_h,w_h\rangle_{\mathcal{T}_h}&=(\hat f,w_h)_\Omega & \forall w_h\in W_h
    \end{array}\right.,
\end{equation}
and
\begin{equation}\label{discppert}
    \left\{\begin{array}{rll}
        \epsilon\langle u_h,v_h\rangle_\Omega+a_h(\lambda_h,v_h) &= 0 & \forall v_h \in V_{0h} \\
         a_h(u_h,w_h)-\gamma^2\langle \lambda_h,w_h\rangle_\Omega&=(\hat f,w_h)_\Omega+(\hat\phi,w_h)_{\Gamma_0} & \forall w_h\in V_{1h}
    \end{array}\right.,
\end{equation}
and for the stabilized scheme (\ref{dissysL2}):
\begin{equation}\label{disL2pert}
    \left\{\begin{array}{rll}
        \epsilon\langle u_h,v_h\rangle_\Omega+a_h(\lambda_h,v_h)+ h^{-2}(u_h,v_h)_\omega &= h^{-2}(\hat q,v_h)_\omega & \forall v_h \in V_h \\
         a_h(u_h,w_h)-\gamma^2s^*(\lambda_h,w_h)&=(\hat f,w_h)_\Omega & \forall w_h\in W_h
    \end{array}\right..
\end{equation}
We assume that \thref{infsup} and \thref{consistency} are satisfied by these systems, then a perturbed version of \thref{conv} and \thref{L2conv} can be derived
\begin{co}\thlabel{pert}
    Let $G\subset \Omega$ satisfies conditions in \thref{tbi}. Let $u\in H^s(\Omega)$ with $2\le s\le k+1$, and $u_h$ be the solutions to the perturbed systems (\ref{disucppert}) or (\ref{discppert}). Let $0<\epsilon<1$ and $h = \epsilon^\frac{1}{2s-2}$. Then there exists some constant $\tau\in(0,1)$ and $C>0$, such that for the unique continuation problem
    \begin{equation*}
        ||u-u_h||_{H^1(G)} \le C\left(\epsilon^\frac{\tau}{2}||u||_{H^s(\Omega)}+\epsilon^{-\frac{1-\tau}{2}}(||\hat f-f||_\Omega+||\hat q-q||_\omega)\right),
    \end{equation*}
    and for the Cauchy problem
    \begin{equation*}
        ||u-u_h||_{H^1(G)} \le C\left(\epsilon^\frac{\tau}{2}||u||_{H^s(\Omega)}+\epsilon^{-\frac{1-\tau}{2}}(||\hat f-f||_\Omega+||\hat \phi-\phi||_{\Gamma_0})\right),
    \end{equation*}
\end{co}
\begin{co}\thlabel{pertL2}
    Under similar assumptions as for Corollary  \thref{pert}, but with  $h = \epsilon^\frac{1}{2s}$, there holds for (\ref{disL2pert}), 
    \begin{equation*}
        ||u-u_h||_{H^1(G)} \le C\left(\epsilon^\frac{\tau}{2}||u||_{H^s(\Omega)}+\epsilon^{-\frac{1-\tau}{2}}(||\hat f-f||_\Omega+||h^{-1}(\hat q-q)||_\omega)\right).
    \end{equation*}
\end{co}
\begin{remark}
    We omitted the proofs for stability under perturbation, which can be derived simply by repeating the proofs for schemes without perturbation. We only mention that according to \thref{pertL2}, the scheme (\ref{disL2pert}) is more sensitive to perturbation comparing to scheme (\ref{disucppert}). This is due to the increasing weight of the interior data $h^{-2}(u_h,v_h)_\omega$in (\ref{disL2pert}), which shows a trade-off between convergence rate and stability.
\end{remark}

\section{Feasible finite element spaces and stability analysis}
In this section we present the explicit construction of finite element spaces for which \thref{infsup} and \thref{consistency} are satisfied. We begin by defining some basic notations. Let $\Omega$, $\mathcal{T}$, $\mathcal{F}$ be defined as in \ref{prelim}. Note that since we do not need to distinguish discrete spaces and continuous spaces, subscript $h$ is omitted for convenience. Let $\Gamma$ be a subset of $\partial\Omega$. Define $\mathcal{V}$ to be the collection of all vertices of simplices in $\mathcal{T}$, $\mathcal{V}_\Gamma$ to be those lying on $\Gamma$, and $\mathcal{V}_I$ to be those lying inside $\Omega$. Similarly the set of facets lying on $\Gamma$ is denoted by $\mathcal{F}_\Gamma$ and those lying inside by $\mathcal{F}_I$. For $d=3$, we also denote the set of all edges of simplices (lying on $\Gamma$, inside $\Omega$) by $\mathcal{E}$ ($\mathcal{E}_\Gamma$, $\mathcal{E}_I$). Note for simplicity, while considering the Cauchy problem, we always assume $\mathcal{F}_{\Gamma_0}\cap\mathcal{F}_{\Gamma_1}=\emptyset$. Recall that $V_h^k$ is the conforming finite element space. We define a Lagrange basis for $V_h^k$ following the notations in \cite{ciarlet2018family}. Let the nodal points on the reference element $\hat{K}$ to be the equidistant nodes with respect to barycentric coordinates $\{\lambda_0,\lambda_1,...,\lambda_d\}$. Denote $\mathcal{N}^k$ the collection of nodes for global shape functions. The Lagrange functions forming a basis of $V_h^k$ indexed by nodal point $N$'s are denoted by $B_N^k$, which satisfies
\begin{equation*}
    \forall N,N' \in \mathcal{N}^k, B_N^k(N')=\delta_{N,N'},
\end{equation*}
where $\delta_{N,N'}$ is the Kronecker delta. Now we define the following subspaces of $V_h^k$. Let $\mathcal{X}$ be a subset of $\mathcal{T},\mathcal{F}$ or $\mathcal{E}$, and $\mathcal{Y}$ be a subset of $\mathcal{V}$, then
\begin{equation}\label{subspaces}
    \begin{aligned}
        V^k_{h,\mathcal{X}} &:=\Span \{B^k_N: N\in X^o\cap \mathcal{N}^k,\forall X\in \mathcal{X}\}\\
        V^k_{h,\mathcal{Y}} &:=\Span\{B^k_N: N\in \mathcal{Y}\cap\mathcal{N}^k\}\\
    \end{aligned}.
\end{equation}
Using the definition in (\ref{subspaces}), we categorize the following spaces.
\begin{prop}
    The following direct sum decomposition hold
    \begin{equation}
        \begin{aligned}
            V_h^k &= V^k_{h,\mathcal{T}}\oplus V^k_{h,\mathcal{F}}\oplus V^k_{h,\mathcal{E}}\oplus V^k_{h,\mathcal{V}}\\
            W_h^k &= V^k_{h,\mathcal{T}}\oplus V^k_{h,\mathcal{F}_I}\oplus V^k_{h,\mathcal{E}_I}\oplus V^k_{h,\mathcal{V}_I}\\
            V^k_{h0} &= V^k_{h,\mathcal{T}}\oplus V^k_{h,\mathcal{F}_I}\oplus V^k_{h,\mathcal{F}_{\Gamma_1}}\oplus V^k_{h,\mathcal{E}_I\cup\mathcal{E}_{\Gamma_1}}\oplus V^k_{h,\mathcal{V}_I\cup\mathcal{V}_{\Gamma_1}}\\
            V^k_{h1} &= V^k_{h,\mathcal{T}}\oplus V^k_{h,\mathcal{F}_I}\oplus V^k_{h,\mathcal{F}_{\Gamma_0}}\oplus V^k_{h,\mathcal{E}_I\cup\mathcal{E}_{\Gamma_0}}\oplus V^k_{h,\mathcal{V}_I\cup\mathcal{V}_{\Gamma_0}}
        \end{aligned}
    \end{equation}
\end{prop}
The following technical results will be useful to prove stability.
\begin{prop}[Inverse trace inequality]\thlabel{invtrace}
    Let $p\in V^k_{h,\mathcal{F}}$, then there exists some constant $C>0$ independent of $h$ (but depends on $k$), such that
    \begin{equation*}
        ||p||_\Omega \le C\sum_{F\in\mathcal{F}}||h^\frac{1}{2}p||_F
    \end{equation*}
\end{prop}
\begin{proof}
    consider the restriction of $p$ in any simplex $K$ and pull back to the reference simplex $\hat{K}$, Denote $\mathbf{G}_d$ the Gram matrix under $L^2(K)$ norm and $N_1, N_2,...,N_l$ be the nodes with respect to $\hat{K}$, then 
    \begin{equation}\label{L2l2}
          ||p||_{\hat{K}}^2=(p(N_1),p(N_2),...,p(N_l))G_d(p(N_1),p(N_2),...,p(N_l))^T\le C\sum_{N_i\in \partial \hat{K}}p^2(N_i).
    \end{equation}
  Since the equidistant Lagrange nodes form a basis on any $(d-1)$ dimensional face $F$ of $\hat{K}$, the respective Gram matrix has a smallest singular value. Thus we have on each face $F$:
  \begin{equation}\label{l2L2}
      \sum_{N_i\in F}p^2(N_i) \le C||p||_F^2.
  \end{equation}
    Then by (\ref{L2l2}) and (\ref{l2L2}), we have after standard scaling process
    \begin{equation*}
        ||p||^2_K \le Ch||p||^2_{\partial K}.
    \end{equation*}
    Then by summing over all $K$'s we get the desired result.
    
\end{proof}

\begin{prop}\thlabel{match}
    Let $K\in \mathcal{T}$ be a $d$-dimensional simplex and  $F\in\mathcal{F}$ be a $(d-1)$ dimensional face, then\\
    (1) for any $q\in\mathbf{P}_k(F)$, there exists some function $p\in V^{k+d}_{h,\{F\}}$, such that $\forall q'\in \mathbf{P}_k(F)$,
    \begin{equation*}
        (p,q')_F=(q,q')_F;
    \end{equation*}
    (2) for any $q(x)\in\mathbf{P}_k(K)$, there exists some function $p\in V^{k+d+1}_{h,\{K\}}$, such that $\forall q'\in \mathbf{P}_k(K)$,
    \begin{equation*}
        (p,q')_K=(q,q')_K,
    \end{equation*}
    where $\mathbf{P}_k(D)$ represents the space of polynomials with order less than or equal to $k$ on domain $D$.\\
    (3) for $p,q$ satisfying conditions in (1) or (2), there exists some constant $C>0$ depending on $k$ and $d$ but independent of the measure of $D$ such that
    \begin{equation*}
        ||p||_D \le C||q||_D,
    \end{equation*}
    where $D = K$ or $F$. 
\end{prop}

\begin{proof}
    We give the proof for (2) and (1) follows similarly. Since all spaces considered here are finite dimensional inner product spaces, we introduce the projection from $\mathbf{P}_{k+d+1}$ to $\mathbf{P}_k(K)$, denoted as $T$. We need to show that the restriction of $T$ on $V^{k+d+1}_{h,\{K\}}$ is invertible, which can be derived by showing $\ker T \cap V^{k+d+1}_{h,\{K\}}=\{0\}$ and $\dim V^{k+d+1}_{h,\{K\}}= \dim \mathbf{P}_k(K)$. The latter holds by observing that 
    \begin{equation*}
        \dim V^{k+d+1}_{h,\{K\}}=\binom{k+d}{k}= \dim \mathbf{P}_k(K),
    \end{equation*}
    where the first equality is easily verified by Vandermonde's Identity. Now we only need to show the kernel of $T$ has no non-zero intersection with $V^{k+d+1}_{h,\{K\}}$. Since there are $N = \binom{k+d}{k}$ nodes within $K$, and these nodes can be derived through an affine mapping from the set of equidistant Lagrange nodes for $\mathbf{P}_k(K)$, they are a valid set of nodes to define a Lagrange basis for $\mathbf{P}_k(K)$. We denote the basis under barycentric coordinates $q_i(\mathbf{\lambda}), i\in\{1,...,N\}$, where $\mathbf{\lambda} = (\lambda_0,\lambda_1,...,\lambda_d)$. Then the basis for $V^{k+d+1}_{h,\{K\}}$ are $\prod_{j=0}^d\lambda_j q_i(\mathbf{\lambda})$. Consider $\mathbf{G}$ defined as $\mathbf{G}_{ij}=\int_{K}\prod_{k=0}^d\lambda_k q_iq_jd\lambda$. Then $\mathbf{G}$ can be viewed as the Gram matrix with respect to the basis $q_j$'s and the weighted inner product, and thus $\mathbf{G}$ is invertible. Let the coordinates of a function $p\in V^{k+d+1}_{h,\{K\}}$ be $\mathbf{x}$, then $p\in\ker T$ iff $\mathbf{G}x = \mathbf{0}$. This gives $p=0$. We conclude that the restriction of $T$ on $V^{k+d+1}_{h,\{K\}}$ has a bounded inverse, and by standard scaling process we obtain the estimate in (3).
    
\end{proof}

\subsection{Full conforming spaces and stability}\label{fullconforming}
In this section, we propose a choice of $W_h$ to be the full conforming subspace of $H^1_0(\Omega)$ ($V_1$). Let $d$ be the dimension of the geometry $\Omega$, $V_h (V_{0h}) = V_h^k (V_{0h}^k)$ for $k\ge 1$, and $W_h (V_{1h}) = W_h^{k+d-1} (V_{1h}^{k+d-1})$. Notice that the systems (\ref{dissys}) and (\ref{dissys2}) satisfy \thref{consistency} directly by the weak formulation of the continuous problems. Note that since the discrete spaces are all conforming, we naturally write integration over meshes as integration over the whole domain $\Omega$. In the rest of the section we will prove \thref{infsup} is fulfilled in our setting. 
\begin{lemma}\thlabel{match1}
    Let $u_h\in V_h^k$ (or $V_{0h}^k$) with $k\ge 2$, then there exists functions $w_1\in V_{h,\mathcal{T}}^{k+d-1}$, $w_2\in V_{h,\mathcal{F}_I}^{k+d-1}$ (or $V_{h,\mathcal{F}_I\cup\mathcal{F}_{\Gamma_0}}^{k+d-1}$), such that $\forall p,q\in X_h^k$ satisfying $p|_K\in \mathbf{P}_{k-2}(K), \forall K\in\mathcal{T}$ and $q|_F\in \mathbf{P}_{k-1}(F),\forall F\in \mathcal{F}$ , we have\\
    (1)
    \begin{equation*}
        (w_1,p)_\Omega = \sum_{K\in\mathcal{T}}(h^2\Delta u,p)_K.
    \end{equation*}
    (2) 
    \begin{equation*}
        \sum_{F\in\mathcal{F}_C}(w_2,q)_F = \sum_{F\in\mathcal{F}_C}\int_Fh\llbracket\nabla u_h\cdot n_F\rrbracket q_FdS,
    \end{equation*}
    where $\mathcal{F}_C=\mathcal{F}_I$ or $\mathcal{F}_I\cup \mathcal{F}_{\Gamma_0}$.\\
    (3) There exists some constant $C>0$ independent of mesh size $h$, such that
    \begin{equation}\label{w1L2}
       ||w_1||_{H^1(\Omega)}\le C ||h^{-1}w_1||_\Omega \le Ch||\Delta_h u_h||_\Omega,
    \end{equation}
    and
    \begin{equation}\label{w2L2}
       ||h^{-1}w_2||^2_{\Omega} \le  C\sum_{F\in\mathcal{F}_C}||h^{-\frac{1}{2}}w_2||^2_{F} \le C_2 \sum_{F\in\mathcal{F}_C}||h^\frac{1}{2}\llbracket \nabla u_h\cdot n_F\rrbracket||^2_F.
    \end{equation}
\end{lemma}
\begin{proof}
    (1) and (2) are directly obtained by summing functions in \thref{match} over all simplices (faces), so are the second inequalities in (\ref{w1L2}) and (\ref{w2L2}). The estimate for the first inequalities in (\ref{w1L2}) inverse inequality (\ref{inv}), and the first inequality in (\ref{w2L2}) is by applying inverse inequality the inverse trace inequality in \thref{invtrace}.
    
\end{proof}
\begin{lemma}\thlabel{matchconforming}
   Let $u_h\in V_h^k$ (or $V_{0h}^k$) with $k\ge 2$ be the solution of (\ref{dissys}) or (\ref{dissys2}). Then there exists a function $w_h \in W_h^{k+d-1}$ (or $(V_{1h}^{k+d-1})$), such that for the Cauchy problem
    \begin{equation}\label{upper}
        ||w_h||^2_{H^1(\Omega)} \le C_0 (\mathcal{J}_h(u_h,u_h)+ ||h^\frac{1}{2}\nabla_hu_h\cdot n||^2_{\Gamma_0}+||h\Delta_hu_h||^2_\Omega),
    \end{equation}
    where $n$ is the outward normal vector on $\Gamma_0$, and 
    \begin{equation}\label{lower}
        a(u_h,w_h)\ge \mathcal{J}_h(u_h,u_h)+||h^\frac{1}{2}\nabla_hu_h\cdot n||^2_{\Gamma_0}+||h\Delta_hu_h||^2_\Omega.
    \end{equation}
    For the unique continuation, the $||h^\frac{1}{2}\nabla_hu_h\cdot n||^2_{\Gamma_0}$ on the above estimates should be removed.
\end{lemma}
\begin{proof}
    We only prove it for Cauchy problem. Let $w_h=-\alpha w_1 + \beta w_2$, with $w_1,$ $w_2$ defined as in \thref{match1}. Now we determine $\alpha,\beta$ satisfying the estimates. Note that if $\alpha$ and $\beta$ does not depend on $h$, then the upper bound (\ref{upper}) holds by using (\ref{w1L2}) and (\ref{w2L2}) directly. For (\ref{lower}), we have
    \begin{equation*}
        \begin{aligned}
            a(u_h,w_h) &= (-\Delta_hu_h,w_h)_\Omega+\sum_{F\in\mathcal{F}_I} \int_F \llbracket \nabla u_h\cdot n_F\rrbracket w_hdS+\sum_{F\in\mathcal{F}_{\Gamma_0}} \int_F \llbracket \nabla u_h\cdot n_F\rrbracket w_hdS\\
            &= \alpha ||h\Delta_hu_h||^2_\Omega + \beta \mathcal{J}_h(u_h,u_h)+\beta ||h^\frac{1}{2}\nabla_hu_h\cdot n||^2_{\Gamma_0}-\beta(\Delta_hu_h,w_2)_\Omega,
        \end{aligned}
    \end{equation*}
    where the second equality applies the fact that $w_1=0$ on any $F\in\mathcal{F}$. To estimate the last term on the right-hand side we have
    \begin{equation*}
        \begin{aligned}
            (\Delta_hu_h,w_2)_\Omega &\le ||h^{-1}w_2||_\Omega||h\Delta_hu_h||_\Omega\\
            &\le 2C_2 ||h\Delta_hu_h||_\Omega^2+\frac{1}{2C_2}||h^{-1}w_2||^2_\Omega\\
            &\le 2C_2 ||h\Delta_hu_h||_\Omega^2+\frac{1}{2}\mathcal{J}_h(u_h,u_h)+\frac{1}{2}||h^\frac{1}{2}\nabla_hu_h\cdot n||^2_{\Gamma_0},
        \end{aligned}
    \end{equation*}
    where $C_2$ is the constant defined in (\ref{w2L2}). By combining these two estimates we obtain
    \begin{equation*}
        a(u_h,w_h) \ge (\alpha-2C_2)||h\Delta_hu_h||^2_\Omega + \frac{1}{2}\beta \left(\mathcal{J}_h(u_h,u_h)+||h^\frac{1}{2}\nabla_hu_h\cdot n||^2_{\Gamma_0}\right).
    \end{equation*}
    Now by setting $\beta=2$ and $\alpha \ge 2C_2+1$ we obtain the desired result. Note that $C_2$ is independent of $h$ and so is $\alpha$.
    
\end{proof}

\begin{theorem}\thlabel{infsupforfullconforming}
    With the $W_h$ and $V_{1h}$ chosen in this section, there exists some $\gamma>0$, such that \thref{infsup} is satisfied.
\end{theorem}

\begin{proof}
For simplicity, we give details only for unique continuation problems. Recall from (\ref{A}) that 
    \begin{equation}\label{infsupconforming1}
    \begin{aligned}
        A_{uc}[(u_h,\lambda_h),(u_h,-\lambda_h+w_h)] = &||u_h||^2_\omega +\epsilon||u_h||^2_{H^1(\Omega)}\\&+a(u_h,w_h)+\gamma^2||\lambda_h-w_h||^2_{H^1(\Omega)}\\
        \ge &||u_h||^2_\omega +\epsilon||u_h||^2_{H^1(\Omega)}+\gamma^2||\lambda_h-w_h||^2_{H^1(\Omega)}\\
        &+\mathcal{J}_h(u_h,u_h)+||h\Delta_hu_h||^2_\Omega.
        \end{aligned}
    \end{equation}
    Since 
    \begin{equation*}
    \begin{aligned}
        ||\lambda_h-w_h||^2_\Omega &= ||w_h||^2_{H^1(\Omega)}+||\lambda_h||^2_{H^1(\Omega)}-2\langle w_h,\lambda_h\rangle_\Omega\\
        &\ge -3||w_h||^2_{H^1(\Omega)} +\frac{1}{2}||\lambda_h||^2_{H^1(\Omega)}\\
        &\ge \frac{1}{2}||\lambda_h||^2_{H^1(\Omega)}-3C_0\left(\mathcal{J}_h(u_h,u_h)+ ||h\Delta_hu_h||^2_\Omega\right),
    \end{aligned}
    \end{equation*}
    where the third inequality is from (\ref{upper}). Then by setting $3C_0\gamma^2\le \frac{1}{2}$ and plugging this estimate into (\ref{infsupconforming1}), we have
    \begin{equation*}
        A_{uc}[(u_h,\lambda_h),(u_h,-\lambda_h+w_h)]\ge \frac{1}{2}|||(u_h,\lambda_h)|||_{uc}|||(u_h,-\lambda_h+w_h)|||_{uc}.
    \end{equation*}
    By repetiting this process we get the same result for the Cauchy problem.
    
\end{proof}

\subsection{Reduced conforming spaces in $d\ge3$}
From \thref{match}, it is observed that $w_1$, $w_2$ live in the space 
\begin{equation*}
    V^{k+d-1}_{h,\mathcal{T}} \oplus V^{k+d-1}_{h,\mathcal{F}_C},
\end{equation*}
where $\mathcal{F}_C = \mathcal{F}_I$ or $\mathcal{F}_I\cup \mathcal{F}_{\Gamma_0}$. It is natural to consider taking only this subset of $W_h^{k+d+1}$ or $V_{1h}^{k+d-1}$ to reduce computational cost, because we do not require an optimal finite element interpolation in these two spaces. However, according to (\ref{rbounduc}) and (\ref{rboundcp}), first order interpolation is still required. Thus we propose a reduced conforming space such that \thref{consistency} and \thref{infsup} are satisfied. Let
\begin{equation}\label{reducedconforminguc}
    W_h = V^{k+d-1}_{h,\mathcal{T}}\oplus V^{k+d-1}_{h,\mathcal{F}_I}\oplus V^1_{h,\mathcal{V}_I},
\end{equation}
and
\begin{equation}\label{reducedconformingcp}
    V_{1h} = V^{k+d-1}_{h,\mathcal{T}}\oplus V^{k+d-1}_{h,\mathcal{F}_I}\oplus V^{k+d-1}_{h,\mathcal{F}_{\Gamma_0}}\oplus  V^1_{h,\mathcal{V}_I\cup\mathcal{V}_{\Gamma_0}}.
\end{equation}
It is easy to see the above direct sums hold. Since the space of piecewise linear functions can be defined at degree of freedoms with respect to vertices. The above spaces contains the full space of piecewise linear functions, and thus (\ref{rbounduc}) and (\ref{rboundcp}) are legit. Note that, (\ref{reducedconforminguc}) and (\ref{reducedconformingcp}) coincide with the full conforming spaces when $d=2$, and omit all degree of freedoms on $d'$ dimensional faces for $d' = 1,..,d-2$.
\subsection{Lifted Crouzeix-Raviart spaces} \label{liftedcr}
It can be seen, through the direct sum representation of reduced conforming spaces in (\ref{reducedconforminguc}) and (\ref{reducedconformingcp}), that the degree of freedoms in such spaces serves just the purpose of stability and shows a fair reduction in computational cost while lower order polynomial spaces are used. However, this choice of spaces is not minimal. We will give an example of a non-conforming space in 2D, such that it has less degree of freedom than the reduced conforming spaces. Let
\begin{equation*}
    \mathbf{CR}_{k}= \{v_h\in X^k_h: \int_{F}\llbracket v_h\rrbracket q dS = 0, \forall q\in \mathbf{P}_{k-1}(F), \forall F\in \mathcal{F}_I\},
\end{equation*}
and
\begin{equation}
    \mathbf{CR}_{k,\Gamma}= \{v_h\in \mathbf{CR}_k: \int_{F}v_hq dS = 0, \forall q\in \mathbf{P}_{k-1}(F), \forall F\in \mathcal{F}\backslash\mathcal{F}_\Gamma  \}.
\end{equation}
Specifically, we denote $\mathbf{CR}_{k,0}$ when $\Gamma =\emptyset$ to keep consistency with most literature. The above space are well-known Crouzeix-Raviart space \cite{crouzeix1973conforming} with generalization to general $k$th order polynomials. We restrict ourselves here to the case when $k$ is odd. A basis is constructed through the following face-bubble functions (see \cite[Definition 3.2]{carstensen2022critical} and references therein)
\begin{equation*}
    B_{k,F}^{CR} = \left\{\begin{array}{lc}
       P_k(1-2\lambda_{K,z_{K,F}})  & \text{on $K$ for } K\in\mathcal{T}_F, \\
        0 & \text{else,}
    \end{array}\right.
\end{equation*}
where $\mathcal{T}_F$ is the collection of triangles having $F$ as one of their faces, $P_k$ is the Legendre polynomial of order $k$, $z_{K,F}$ is the vertex in $K$ opposite to $F$, and $\lambda_{K,z_{K,F}}$ be the barycentric coordinate in $K$ with respect to the vertex $z_{K,F}$. It can be observed that 
\begin{equation}\label{BKF}
    \begin{array}{cc}
      B_{k,F}^{CR}|_F = 1;   & \int_{F'} B^{CR}_{k,F}qdx = 0, \forall q\in \mathbf{P}_{k-1}(F'),\forall F'\neq F.
    \end{array}
\end{equation}
Moreover, let $V_{nc,\mathcal{F}_C}^k = \Span\{B^{CR}_{k,F}:F\in \mathcal{F}_C\}$, where $\mathcal{F}_C$ is a subset of $\mathcal{F}$. Then a direct sum decomposition for $\mathbf{CR}_{k,\Gamma}$ is 
\begin{equation}\label{crgamma}
    \mathbf{CR}_{k,\Gamma} = V_{h,\mathcal{T}}^k \oplus V_{h,\mathcal{F}_I}^k\oplus V_{h,\mathcal{F}_\Gamma}^k\oplus V_{nc,\mathcal{F}_I}^k\oplus V_{nc,\mathcal{F}_\Gamma}^k.
\end{equation}
Now we lift the order of polynomials with respect to the degree of freedoms interior to triangles in $\mathbf{CR}_{k,\Gamma}$ by $1$, namely,
\begin{equation}\label{nonconuc}
    W_h = V^{k+1}_{h,\mathcal{T}}\oplus V^{k}_{h,\mathcal{F}_I}\oplus V_{nc,\mathcal{F}_I}^k,
\end{equation}
and
\begin{equation}\label{nonconcp}
    V_{1h} = V^{k+1}_{h,\mathcal{T}}\oplus V^{k}_{h,\mathcal{F}_I}\oplus V^{k}_{h,\mathcal{F}_{\Gamma_0}}\oplus  V_{nc,\mathcal{F}_I}^k\oplus  V_{nc,\mathcal{F}_{\Gamma_0}}^k.
\end{equation}
It is clear that 
\begin{equation*}
\begin{array}{cc}
    \mathbf{CR}_{k,0}\subsetneq W_h\subsetneq \mathbf{CR}_{k+1,0}, &  \mathbf{CR}_{k,\Gamma_0}\subsetneq V_{1h} \subsetneq\mathbf{CR}_{k+1,\Gamma_0}.
\end{array}
\end{equation*}
This shows $W_h$ and $V_{1h}$ contain the conforming space of polynomials of order at most $k$, so (\ref{rbounduc}) and (\ref{rboundcp}) are valid. Moreover, the dimension of (\ref{nonconuc}) and (\ref{nonconcp}) are strictly less than the reduced conforming spaces (\ref{reducedconforminguc}) and (\ref{reducedconformingcp}). This can be shown by observing that 
\begin{equation*}
    \dim{V^{k}_{h,\mathcal{F}_C}\oplus V_{nc,\mathcal{F}_C}^k} = \dim{V^{k+1}_{h,\mathcal{F}_C}},
\end{equation*}
where $\mathcal{F}_C = \mathcal{F}_I$ or $\mathcal{F}_{\Gamma_0}$. Moreover, we show the choice in (\ref{nonconuc}) and (\ref{nonconcp}) fulfills \thref{infsup} and \thref{consistency} in Appendix \ref{veri}.

\subsection{The simplest unregularized schemes}
In this section we discuss the simplest discrete schemes where the $\epsilon-$regularization (cf. (\ref{dissys}), (\ref{dissys2})) may be omitted. The scheme applies the idea in \cite{burman2017stabilized} but a conforming primal space is used, so that a further penalty of jumps on interior facets can be ommited. From now on we need to assume the exact solution $u\in H^2(\Omega)$, which makes the data terms $f_h,\phi_h$ exact. The simplest systems read:
\begin{equation}\label{nonregdissys}
    \left\{\begin{array}{rll}
        a_h(\lambda_h,v_h)+ (u_h,v_h)_\omega &= (q,v_h)_\omega & \forall v_h \in V_h \\
         a_h(u_h,w_h)-\gamma^2\langle \lambda_h,w_h\rangle_{\mathcal{T}_h}&=(f,w_h)_\Omega & \forall w_h\in W_h
    \end{array}\right.,
\end{equation}
and
\begin{equation}\label{nonregdissys2}
    \left\{\begin{array}{rll}
        a_h(\lambda_h,v_h) &= 0 & \forall v_h \in V_{0h} \\
         a_h(u_h,w_h)-\gamma^2\langle \lambda_h,w_h\rangle_\Omega&=(f,w_h)_\Omega+(\phi,w_h)_{\Gamma_0} & \forall w_h\in V_{1h}
    \end{array}\right..
\end{equation}
Let $k=1$, which means $V_h$ ($V_{0h}$) is the conforming piecewise linear functions. Set $W_h = \mathbf{CR}_{1,0}$ ($V_{1h}=\mathbf{CR}_{1,\Gamma_0}$). Note that this is a special case of Section \ref{liftedcr} (these spaces are not lifted because $\mathbf{P}_2(K)$ does not have any interior degree of freedom). Moreover, the validity of this pair can be extended to $\Omega$ in any dimension $d$.

Recall the norms $|||(\cdot,\cdot)|||_{uc}$ and $|||(\cdot,\cdot)|||_{cp}$ (cf. (\ref{triplenormuc}), (\ref{triolenormcp})). Since we no longer use the regularization term $\epsilon ||u_h||^2_{H^1(\Omega)}$ but we still need control its $H^1$-norm, we set the $\epsilon ||u_h||^2_\Omega$ term in the triple norms as a mesh-dependent term $||hu_h||^2_{H^1(\Omega)}$. Indeed, for piecewise affine approximations, $||u_h||_{H^1(\Omega)}$ is bounded by the interior jump of gradients and $||u||_\omega$ (or normal derivative on $\Gamma_0$ for the Cauchy problem), as we will show in \thref{poincare}. Note if \thref{infsup} and \thref{consistency} hold, we can still derive optimal convergence from the schemes, since $||hu_h||^2_{H^1(\Omega)}$ is adapted from the optimal choice of $h$ in \thref{conv}.\\

Now we show that these pairs produce stable and consistent schemes. First, consistency can be verified by redoing the process as in Section \ref{liftedcr} since the proofs for consistency do not depend on dimension. For stability, we have
\begin{prop}\thlabel{simplematch}
    If $u_h\in V_h^1$, then there exists a function $w_h\in \mathbf{CR}_{1,0}$ , such that
    \begin{equation}
    \begin{aligned}
        ||w_h||^2_{H^1(\Omega)} &\le C\mathcal{J}_h(u_h,u_h),\\
        a(u_h,w_h)&\ge \mathcal{J}_h(u_h,u_h).
    \end{aligned}
    \end{equation}
    If $u_h\in V^1_{0h}$, then there exists a function $w_h\in \mathbf{CR}_{1,\Gamma_0}$ , such that
    \begin{equation}
    \begin{aligned}
        ||w_h||^2_{H^1(\Omega)} &\le C(\mathcal{J}_h(u_h,u_h)+||h^\frac{1}{2}\nabla_hu_h\cdot n||^2_{\Gamma_0}),\\
        a(u_h,w_h)&\ge \mathcal{J}_h(u_h,u_h)+||h^\frac{1}{2}\nabla_hu_h\cdot n||^2_{\Gamma_0}.
    \end{aligned}
    \end{equation}
    where $n$ is the outward normal vector on $\Gamma_0$, and 
    
\end{prop}
\begin{proof}
    The proof can be easily adapted from \thref{matchconforming}. The reason is that the first order Crouzeix-Raviart finite element can be easily defined as Lagrange finite element. See \cite{crouzeix1973conforming}.
    
\end{proof}
Now we give a Poincaré type inequalities to finalize verifying \thref{infsup}.
\begin{lemma}\thlabel{poincare}
    If $u_h\in V^1_h$, then there exists some constant $C>0$, such that
    \begin{equation*}
        ||hu_h||^2_{H^1(\Omega)} \le C(\mathcal{J}_h(u_h,u_h)+||u_h||^2_\omega).
    \end{equation*}
    If $u\in V^1_{0h}$, then there exists some constant $C>0$, such that
    \begin{equation*}
        ||hu_h||^2_{H^1(\Omega)} \le C(\mathcal{J}_h(u_h,u_h)+||h^\frac{1}{2}\nabla u_h\cdot n_\gamma||^2_{\Gamma_0}).
    \end{equation*}
\end{lemma}
\begin{proof}
    The proof can be adapted from \cite[Lemma 2]{burman2018solving}.
\end{proof}
With the inequality given above we have
\begin{equation*}
     a(u_h,w_h)+||u_h||^2_\omega \ge C_1\mathcal{J}_h(u_h,u_h) + C_2||hu_h||^2_{H^1(\Omega)}
\end{equation*}
for the unique continuation problem and 
\begin{equation*}
     a(u_h,w_h)\ge C_1(\mathcal{J}_h(u_h,u_h) +||h^\frac{1}{2}\nabla u_h\cdot n_\gamma||^2_{\Gamma_0})+ C_2||hu_h||^2_{H^1(\Omega)}
\end{equation*}
for the Cauchy problem. Then by repeting the steps in proving \thref{infsupforfullconforming}, we can verify \thref{infsup}.
\begin{remark}
    We have assumed $u\in H^2(\Omega)$ in this section. This assumption is necessary as we can see in \thref{poincare}, only $||hu_h||^2_{H^1(\Omega)}$ can be bounded by the remaining terms in (\ref{triplenormuc}) or (\ref{triolenormcp}). If $u\in H^s(\Omega)$ with $s<2$, then we need to have a bound for $||h^{s-1}u_h||^2_{H^1(\Omega)}$ which can not be provided by those terms.
\end{remark}

\begin{figure}[b]
    \centering
    \caption{Convergence rates for different test spaces and oscillation rates}%
    \subfloat[\centering $m = 1$, $n=1$.]{{\includegraphics[width=5.8cm]{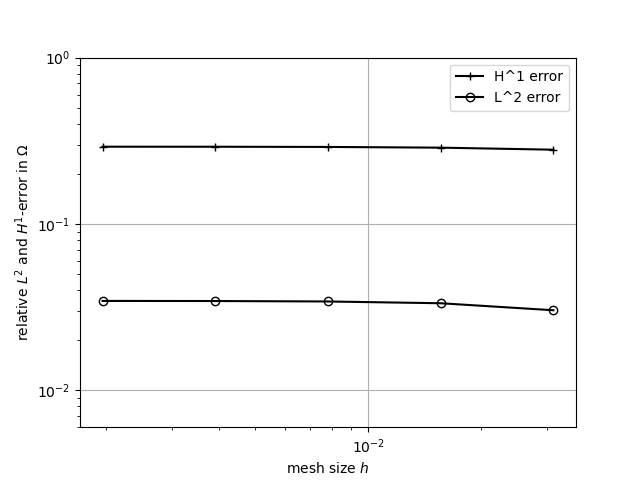} }}%
    \qquad
    \subfloat[\centering $m = 2$, $n=1$.]{{\includegraphics[width=5.8cm]{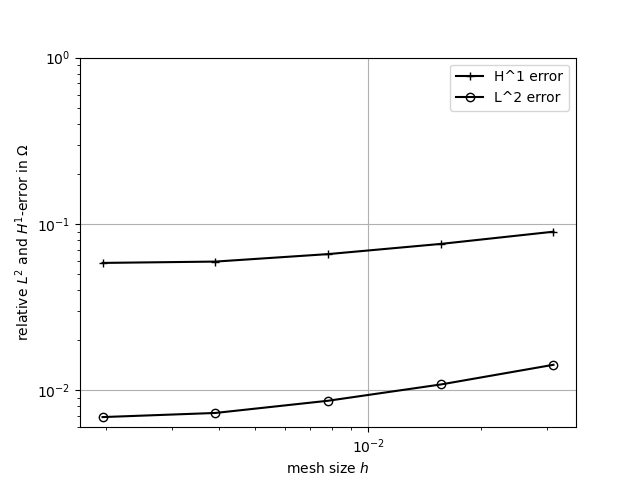} }}%
    \qquad
    \subfloat[\centering $m = 1$, $n=5$.]{{\includegraphics[width=5.8cm]{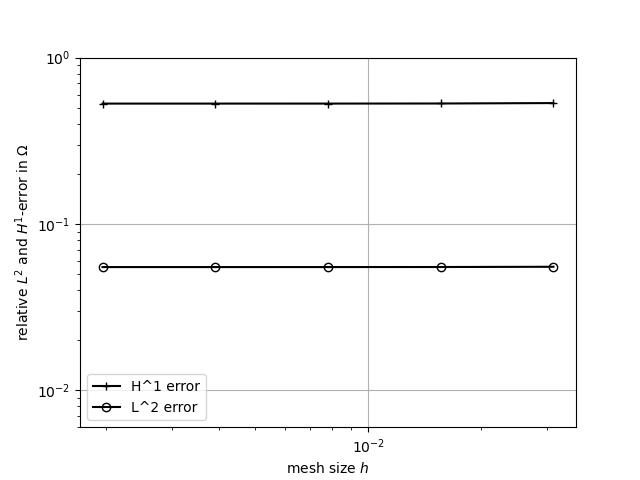} }}%
    \qquad
    \subfloat[\centering $m = 2$, $n=5$.]{{\includegraphics[width=5.8cm]{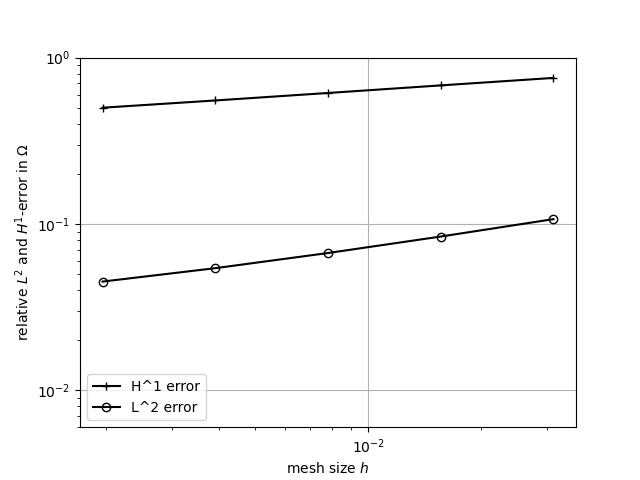} }}%
    \label{convergencerates}
\end{figure}
\section{Numerical experiments}
In this section we present numerical experiments that we conducted to illustrate the scheme we proposed. We set
\begin{equation}
\begin{array}{ccc}
    \Omega = [0,1]\times [0,1] & \text{and}  & \Gamma_0 = (\{0\}\times [0,1])\cup ([0,1]\times \{0\})
\end{array}
\end{equation}
The function we use to approximate is that in the classic Hardamard example $u = \frac{1}{n^2}\sin{x}\sinh{y}$, which is the solution to the Cauchy problem
\begin{equation}\label{hardamardex}
    \left\{\begin{array}{rlc}
         -\Delta u& =0 & \texttt{in $\Omega$} \\
         u &=0 & \texttt{on $\Gamma_0$}\\
         \nabla u\cdot \nu &= -\frac{1}{n}\sinh{ny} & \texttt{on $\{0\}\times [0,1]$}\\
         \nabla u\cdot \nu &= -\frac{1}{n}\sin{nx} & \texttt{on $[0,1]\times \{0\}$}
    \end{array}.
    \right.
\end{equation}

We use the scheme described in (\ref{dissys2}) to numerically approach the solution to (\ref{hardamardex}) with different oscillation rate $n$. We use the full conforming space described in Section \ref{fullconforming} by setting $V_{0h}=V^1_{0h}$ and $V_{1h} = V^m_{1h}$ (cf. (\ref{H1spaces}) with different $m$. Moreover, we set the Tikhonov regularization parameter $\epsilon=10^{-4}$ and choose suitable mesh size $h$ such that the scheme does not achieve the precision determined by $\epsilon$, to compare the convergence behavior of different test spaces. In Figure \ref{convergencerates}, we observe that the $L^2$ error and $H^1$ error in the whole domain $\Omega$ converge when $m=2$ but not when $m=1$. This is consistent with the stable pairs given in Section \ref{fullconforming}. Note that according to \thref{conv}, the convergence rate is only logarithmic with the estimate 
\begin{equation*}
        ||u-u_h||_{\Omega} \le \frac{C}{(-\log{h})^\tau}||u||_{H^2(\Omega)}.
    \end{equation*}

\begin{figure}[h]
    \centering
    \caption{Error comparison for different test spaces. Observe different scales.}%
    \subfloat[\centering  $m=1$.]{{\includegraphics[width=3.6cm]{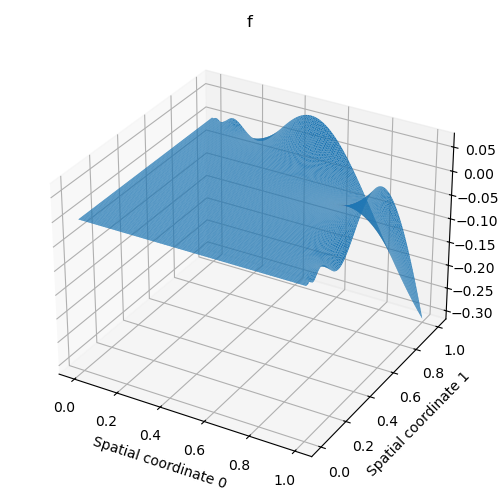} }}%
    \qquad
    \subfloat[\centering  $m=2$.]{{\includegraphics[width=3.6cm]{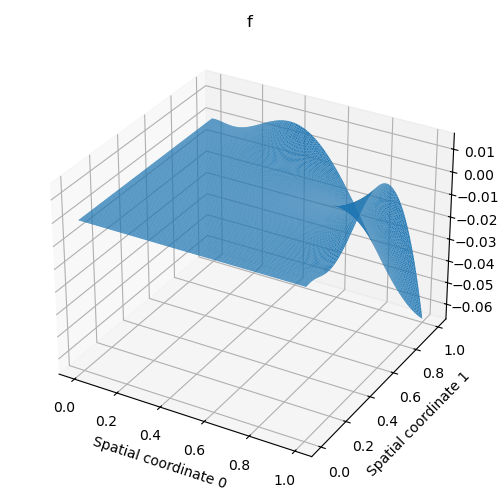} }}%
    \qquad
    \subfloat[\centering  $m=3$.]{{\includegraphics[width=3.6cm]{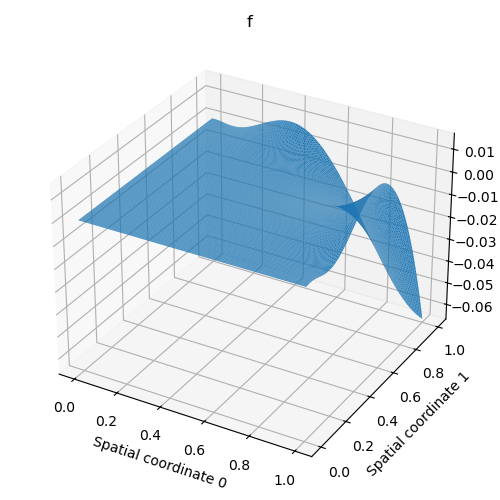} }}%
    \label{surface_plots}
\end{figure}
In Figure \ref{surface_plots}, we compare the error functions $u-u_h$ when choosing conforming spaces with different order of polynomials, as $V_{0h}^1-V_{1h}^m$ pair. In this example we fix the exact solution $u$ defined in (\ref{hardamardex}) with $n=1$ and use the mesh with $128\times 128$ (with respect to $x$, $y$) elements. We observe that the $\mathbf{P}_1$-$\mathbf{P}_2$ pair does provide additional stability we need than $\mathbf{P}_1$-$\mathbf{P}_1$ pair. The $\mathbf{P}_1$-$\mathbf{P}_3$ pair is also stable, but does not provide additional error reduction comparing to $\mathbf{P}_1$-$\mathbf{P}_2$ pair. This is consistent with the theoretical result in Section \ref{fullconforming}.\\

It can be also seen from Figure \ref{surface_plots} that the error function is relatively large near $(x,y) = (1,1)$. This is generally why only the logarithmic convergence is guaranteed by the scheme. However, if we restrict ourselves in an interior region, a Hölder type convergence rate can be expected. In Figure \ref{interiorconv}, we set the interior area $G = [0,0.8]\times [0,0.5]$, $n=5$ and $\epsilon = 10^{-4}$. It can be observed while the $\mathbf{P}_1$-$\mathbf{P}_1$ pair still shows no convergence while the $\mathbf{P}_1$-$\mathbf{P}_2$ pair presents a convergence rate of about $h^{0.3}$ under $H^1$ norm.

\begin{figure}[t]
    \centering
    \caption{Convergence rates comparison in the interior region $G$}%
    \subfloat[\centering  $m=1$.]{{\includegraphics[width=5.8cm]{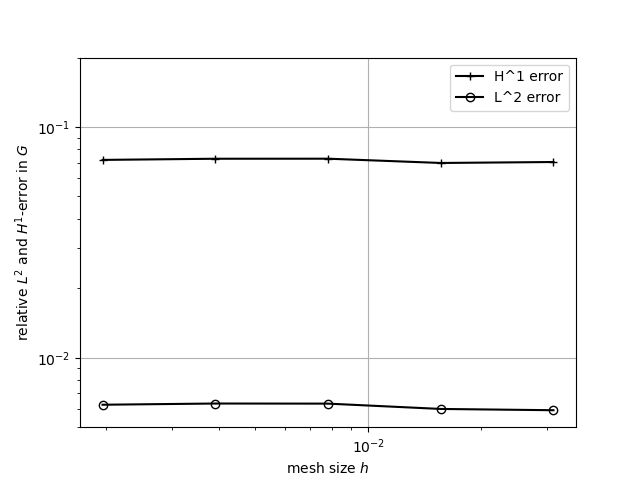} }}%
    \qquad
    \subfloat[\centering  $m=2$.]{{\includegraphics[width=5.8cm]{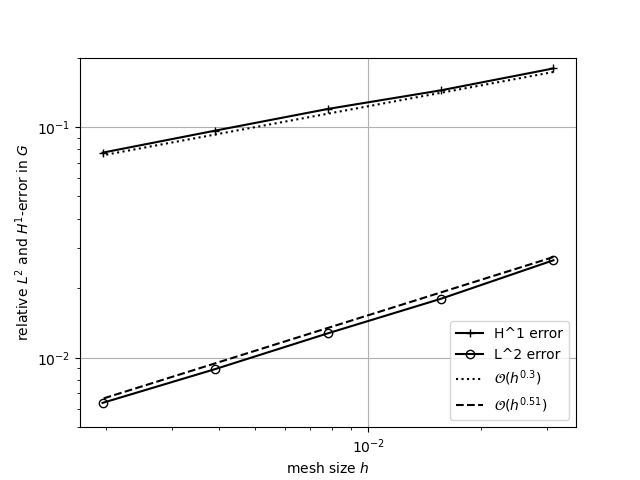} }}%
    \qquad
    \label{interiorconv}
\end{figure}
\appendix
\section{Equivalence of norms}\label{equinorms}
Now we give the proof that to bound $d(u_h-u)$ (c.f. (\ref{datacp})), it is equivalent to bound the operator norm for $a(u-u_h,v)$ on $V_1'$. The idea to prove this statement is to construct the Generalized Cauchy problem in the form (\ref{generalcp}) that $u_h-u$ is a solution to, and then bound the right-hand side with the operator norm of $a(u-u_h,v)$ on $V_1'$.\\

First, define $l_h: V_1'\to V_1$:
\begin{equation}
    l_h(v) = a(u_h-u,v).
\end{equation}
Consider a different inner product $(\cdot,\cdot)_{V_1}$ (other than the natural $H^1$ inner product) on $V_1$ defined as:
\begin{equation}
    (v_1,v_2)_{V_1} = \langle v_1,v_2\rangle_\Omega + (v_1,v_2)_{\Gamma_0}.
\end{equation}
Then by Riesz representation, there exists $\tilde v\in V_1$, such that
\begin{equation}
    l_h(v) = \langle \tilde v,v\rangle_\Omega + (\tilde v,v)_{\Gamma_0}, \forall v\in V_1.
\end{equation}
Taking $v = \tilde v$ we have
\begin{equation}\label{tildevbound}
    ||\tilde v||_{H^1(\Omega)}^2+||\tilde v||^2_{\Gamma_0}\le ||l_h||_{V_1'}||\tilde v||_{H^1(\Omega)}.
\end{equation}
However, we could also see that $u-u_h$ is a solution to the problem 
\begin{equation}
    \int_\Omega \nabla (u-u_h)\cdot\nabla vdx= \int_{\Gamma_0}\tilde v vdS +\int_\Omega \tilde v v + \nabla \tilde v \cdot \nabla v dx.
\end{equation}
This equation satisfies the problem format in (\ref{generalcp}) with the correspondence
\begin{equation*}
    \begin{aligned}
        f & = \tilde v\\
        F &= \nabla \tilde v\\
        \phi &= \tilde v|_{\Gamma_0}.
    \end{aligned}
\end{equation*}
This gives that by (\ref{datacp})
\begin{equation*}
\begin{aligned}
    d(u_h-u) &= ||\tilde v||_\Omega + ||\nabla \tilde v||_\Omega + ||\tilde v||_{\Gamma_0}\le C||l_h||_{V_1'},
\end{aligned}
\end{equation*}
where the inequality is derived feom (\ref{tildevbound}).

\section{Verification of feasibility for lifted Crouzeix-Raviart spaces}\label{veri}

\begin{lemma}[Inverse trace inequality for lifted Crouzeix-Raviart space]\thlabel{invtracecr}
    Let $p\in V^k_{h,\mathcal{F}}\oplus V_{nc,\mathcal{F}}^k$, then there exists some constant $C>0$ independent of $h$ (but depends on $k$), such that
    \begin{equation*}
        ||p||_\Omega \le C\sum_{F\in\mathcal{F}}||h^\frac{1}{2}p||_F
    \end{equation*}
\end{lemma}
\begin{proof}
    The idea follows from \thref{invtrace}. Consider the pull-back of the restriction of $p$ on some triangle $K$. We have
    \begin{equation*}
        p = \sum_{i=0}^2 \left(c_{F_i} B^{CR}_{k,F_i} +\sum_{j=1}^{k-1}c_{N_{j,F_i}}B^k_{N_{j,F_i}}\right),
    \end{equation*}
    where $N_{j,F_i}$ represents the $j$-th node on the face $F_i$. By the invertibility of the Gram matrix there exists some constant $C>0$, such that
    \begin{equation}\label{K<L2}
        ||p||^2_{\hat{K}} \le C \sum_{i=0}^2 c^2_{F_i}+\sum_{j=1}^{k-1}c^2_{N_{j,F_i}}.
    \end{equation}
    Moreover, on each face $F_i$, we have
    \begin{equation}\label{pFi}
        p = \sum_{i=0}^2 \left(c_{F_i} B^{CR}_{k,F_i} +\sum_{j=1}^{k-1}c_{N_{j,F_i}}B^k_{N_{j,F_i}} + (c_{F_l}-c_{F_m})B^{CR}_{k,F_l}\right),
    \end{equation}
    where $l,m\in\{0,1,2\}$ are the two indices other than $i$. This equality has used the fact that the restriction of $B^k_{N_{j,F_l}}$ and $B^k_{N_{j,F_m}}$ are Legendre polynomials with different sign. Moreover, by (\ref{BKF}), the above $k+1$ functions are linear independent since the linear combination of the first $k$ functions takes the same value at both end points of $F_i$, while the last function takes opposite values. This shows that the Gram matrix with respect to these functions on $F_i$ is invertible, and there exists some constant $c>0$, such that
    \begin{equation}\label{L2<F}
        ||p||^2_{F_i} \ge c\sum_{i=0}^2 \left(c_{F_i}^2 +\sum_{j=1}^{k-1}c_{N_{j,F_i}}^2 + (c_{F_l}-c_{F_m})^2\right).
    \end{equation}
Then by combining (\ref{K<L2}) and (\ref{L2<F}) and the standard scaling process we obtain the desired result.

\end{proof}
\begin{theorem}\thlabel{infsupnoncon}
     With the $W_h$ and $V_{1h}$ chosen in (\ref{nonconuc}) and (\ref{nonconcp}), there exists some $\gamma>0$, such that \thref{infsup} is satisfied.
\end{theorem}
\begin{proof}
    To prove \thref{infsupnoncon}, we only need to prove there exists some function $w_2\in V^{k}_{h,\mathcal{F}_I}\oplus V_{nc,\mathcal{F}_I}^k$ (or $V^{k}_{h,\mathcal{F}_I}\oplus V^{k}_{h,\mathcal{F}_{\Gamma_0}}\oplus  V_{nc,\mathcal{F}_I}^k\oplus  V_{nc,\mathcal{F}_{\Gamma_0}}^k$), such that \thref{match1} holds (since the same $w_1$ can be chosen as in \thref{match1}). We will show here the proof when $w_2\in V^{k}_{h,\mathcal{F}_I}\oplus V_{nc,\mathcal{F}_I}^k$.
    
    First, consider $\forall F\in \mathcal{F}_I$. Let $V_F\subset V^{k}_{h,\mathcal{F}_I}$ be the collection of functions which are not identically $0$ on $F$, then 
    \begin{equation*}
        V_F = V^{k}_{h,\{F\}}\oplus \Span\{B^{CR}_{k,F'}: F'\subset\partial K,\forall K\in \mathcal{T}_F  \}.
    \end{equation*}
    Consider the restriction of $V^{k}_{h,\{F\}}\oplus \{B^{CR}_{k,F}\}$ on $F$. By the same argument as we derive (\ref{L2<F}) we obtain that these functions are linear independent along with the Legendre polynomial of order $k$, so its projection on $\mathbf{P}_{k-1}(F)$ is invertible. This gives: $\forall q\in \mathbf{P}_{k-1}(F)$, there exists some $p\in V^{k}_{h,\{F\}}\oplus \{B^{CR}_{k,F}\}$, such that
    \begin{equation}\label{wmatch}
        (p,q')_F = (q,q')_F, \forall q'\in \mathbf{P}_{k-1}(F).
    \end{equation}
    Now consider $\llbracket \nabla u_h\cdot n_F\rrbracket$. Since $u\in V_h^k$, $\llbracket \nabla u_h\cdot n_F\rrbracket|_F\in \mathbf{P}_{k-1}(F) $ for any $F$. We choose $w_2$ by choosing the coordinates in $V^{k}_{h,\{F\}}\oplus \{B^{CR}_{k,F}\}$ to guarantee (\ref{wmatch}). Note that in this way we have just used out all degree of freedoms in $\oplus_{F\in\mathcal{F}_I}V_F$. Then $\forall F\in \mathcal{F}_I$, $w_2|_F = p+p'$, where 
    \begin{equation*}
        \begin{array}{cc}
           p\in  p\in V^{k}_{h,\{F\}}\oplus \{B^{CR}_{k,F}\}, &  p'\in \Span\{B^{CR}_{k,F'}: F'\neq F, F'\subset\partial K,\forall K\in \mathcal{T}_F  \}.
        \end{array}
    \end{equation*}
Moreover, we have
\begin{equation*}
    (p,\llbracket \nabla u_h\cdot n_F\rrbracket)_F = ||h^\frac{1}{2}\llbracket \nabla u_h\cdot n_F\rrbracket||_F^2,
\end{equation*}
and by the second equality in (\ref{BKF})
\begin{equation*}
    (p',\llbracket \nabla u_h\cdot n_F\rrbracket)_F =0.
\end{equation*}
This has proven part (2) as in \thref{match1}. Part (3) in \thref{match1} can also be derived through a direct application of inverse inequality (\ref{inv}) and \thref{invtracecr} .

\end{proof}
Now we show \thref{consistency} is also satisfied by the schemes. We begin this by giving a set of global degree of freedoms for $W_h$ and $V_{1h}$ (cf. (\ref{nonconuc}), (\ref{nonconcp})).
\begin{prop}\thlabel{dofs}
    A set of degree of freedoms is given by taking equi-distance Lagrange nodes of $k+1$-th order in $K^o$ for all $K\in \mathcal{T}$ and taking $k$-th order Gauss-Legendre nodes on $F^o$ for all $F\in \mathcal{F}_I$ (or $\mathcal{F}_I\cup \mathcal{F}_{\Gamma_0}$).
\end{prop}

\begin{proof}
    It is easy to verify that the dimensions of dofs given in \thref{dofs} coincide with the dimensions of (\ref{nonconuc}) or (\ref{nonconcp}). We only need to check the unisolvence. Taking $W_h$ as an example. Let $p\in W_h$, $K\in \mathcal{T}$ be any triangle, we have from (\ref{pFi}) that for each face $F_i\subset \partial K$
    \begin{equation*}
        p = \sum_{i=0}^2 \left(c_{F_i} B^{CR}_{k,F_i} +\sum_{j=1}^{k-1}c_{N_{j,F_i}}B^k_{N_{j,F_i}} + (c_{F_l}-c_{F_m})B^{CR}_{k,F_l}\right).
    \end{equation*}

This is because any function in $V^k_{h,\mathcal{T}}$ vanishes on any $F\in \mathcal{F}$. Now let $p=0$ on all Gauss-Legendre points on each $F_i$. This means $p|_{F_i}$ is a Legendre polynomial of order $k$, and thus a multiple of $B^{CR}_{k,F_l}|_{F_i}$. Since this works for all $F_i\subset \partial K$ such that $B^{CR}_{k,F_i}\in W_h$, we conclude $c_{F_i} = c_{N_{j,F_i}}=0$ for all such $F_i$. Moreover. $p=0$ at all interior nodes then gives the coefficient with respect to $V^k_{h,\mathcal{T}}$ is $0$, and thus $p\equiv 0$.

\end{proof}
Now we prove \thref{consistency} is also satisfied by the chosen spaces. First, we introduce some technical lemmas.
\begin{definition}
    Let $J_{h0}^{av}:X^k_h\to H_0^1(\Omega)$ be defined as in \cite[Section 6.2]{ern2017finite} with the degree of freedoms defined in \thref{dofs}, then there is some constant $C>0$, such that $\forall w_h\in W_h$,
    \begin{equation}\label{faceest}
        |w_h-J_{h0}^{av}w_h|_{H^m(\mathcal{T})}\le Ch^{\frac{1}{2}-m}\sum_{F\in \mathcal{F}}||\llbracket w_h\rrbracket||_F,
    \end{equation}
    where $m\in \{0:k+1\}$ and $\llbracket w_h \rrbracket $ is defined as $w_h$ when $F\in \mathcal{F}_{\partial \Omega}$.
\end{definition}
The details of the boundary prescription averaging operator can be found in \cite[Section 6.2]{ern2017finite} and the interpolation property we use is adapted from \cite[Lemma 6.2]{ern2017finite}.
\begin{lemma}
    Let $w_h\in W_h$, then there exists some constant $C>0$ such that $\forall K\in \mathcal{T}$ and $F\subset \partial K$,
    \begin{equation}\label{facepoincare}
        ||\llbracket w_h\rrbracket||_{F} \le C h^\frac{1}{2}||w_h||_{H^1(K)}.
    \end{equation}
    Moreover, for $m=0,1$,
    \begin{equation}\label{averageinterpo}
        |w_h-J_{h0}^{av}w_h|_{H^m(\mathcal{T})} \le C h^{1-m}||w_h||_{H^1(\mathcal{T})}.
    \end{equation}
\end{lemma}
\begin{proof}
    (\ref{facepoincare}) is the Poincaré–Steklov on faces (see, for example \cite[Lemma 36.8]{ern2021finite2}) with observing that $\llbracket w_h\rrbracket$ has vanishing integral on all faces $F\in\mathcal{F}$. (\ref{averageinterpo}) is derived by using (\ref{faceest}) and then (\ref{facepoincare}).
    
\end{proof}
\begin{theorem}\thlabel{consistencynoncon}
     With the $W_h$ and $V_{1h}$ chosen in (\ref{nonconuc}) and (\ref{nonconcp}), then \thref{consistency} is satisfied.
\end{theorem}
\begin{proof}
We prove the theorem for the unique continuation problem. Let $w_h\in W_h$. We have $a(u,J_{h0}^{av}w_h)-(f_h,J_{h0}^{av}w_h)_\Omega$ by the definition of $f_h$ (cf. (\ref{f_h})). Let $r_h = w_h-J_{h0}^{av}w_h$ and $k\ge s-1$, we have
\begin{equation}
\begin{aligned}
    a(u,w_h)-(f_h,w_h)_\Omega &= a(u,r_h)-(f_h,r_h)_\Omega \\
    & = a(u-\pi^k_{sz}u,r_h)+a(\pi^k_{sz}u,r_h)-(f_h,r_h)_\Omega \\
    &\le Ch^{s-1}||u||_{H^s(\Omega)}||w_h||_{H^1(\mathcal{T})}+a(\pi^k_{sz}u,r_h)-(f_h,r_h)_\Omega,
    \end{aligned}
\end{equation}
where we use the optimal interpolation error for the Scott-Zhang interpolant and (\ref{averageinterpo}) in the last inequality. Now we only need to bound $a(\pi^k_{sz}u,r_h)-(f_h,r_h)_\Omega$. We have by integration by parts
\begin{equation*}
    \begin{aligned}
a(\pi^k_{sz}u,r_h)-(f_h,r_h)_\Omega = &\sum_{K\in\mathcal{T}}\int_K-r_h\Delta \pi^k_{sz}udx - (f_h,r_h)_\Omega\\
&+\sum_{F\in\mathcal{F}_I}\int_F\{r_h\}\llbracket \nabla \pi^k_{sz}u\cdot n_F\rrbracket dS \\
&\lesssim (||h(\Delta_h\pi^k_{sz}u-f_h)||_\Omega+\mathcal{J}_h(\pi^k_{sz}u,\pi^k_{sz}u))||w_h||_{H^1(\mathcal{T})},
    \end{aligned}
\end{equation*}
where trace inequality (\ref{tc2}) and interpolation estimate (\ref{averageinterpo}) are applied for the face quantity. Note the face integrals are counted only in the interior since for $\llbracket \nabla \pi^k_{sz}u\cdot n_F\rrbracket \in \mathbf{P}_{k-1}(F)$ on any $F\in \mathcal{F}$ and $\int_{\partial\Omega}w_h\llbracket \nabla \pi^k_{sz}u\cdot n_F\rrbracket dS=0$. We conclude by using \cite[Proposition 3.2] {BL24} that 
\begin{equation*}
    ||h(\Delta_h\pi^k_{sz}u-f_h)||_\Omega+\mathcal{J}_h(\pi^k_{sz}u,\pi^k_{sz}u)\le Ch^{s-1}||u||_{H^s(\Omega)}.
\end{equation*}

The proof for the Cauchy problem is similar but requires modifications of the average operator to match the boundary condition satisfied by $V_1$.

\end{proof}

\bibliography{bibliography}

\begin{bibdiv}
\begin{biblist}

\bib{alessandrini2009stability}{article}{
      author={Alessandrini, Giovanni},
      author={Rondi, Luca},
      author={Rosset, Edi},
      author={Vessella, Sergio},
       title={The stability for the {Cauchy} problem for elliptic equations},
        date={2009},
     journal={Inverse problems},
      volume={25},
      number={12},
       pages={123004},
}

\bib{Babuška1971}{article}{
      author={Babu{\v{s}}ka, Ivo},
       title={Error-bounds for finite element method},
        date={1971Jan},
        ISSN={0945-3245},
     journal={Numerische Mathematik},
      volume={16},
      number={4},
       pages={322\ndash 333},
         url={https://doi.org/10.1007/BF02165003},
}

\bib{bourgeois2005mixed}{article}{
      author={Bourgeois, Laurent},
       title={A mixed formulation of quasi-reversibility to solve the {Cauchy}
  problem for {Laplace's} equation},
        date={2005},
     journal={Inverse problems},
      volume={21},
      number={3},
       pages={1087},
}

\bib{bourgeois2018mixed}{article}{
      author={Bourgeois, Laurent},
      author={Recoquillay, Arnaud},
       title={A mixed formulation of the {Tikhonov} regularization and its
  application to inverse {PDE} problems},
        date={2018},
     journal={ESAIM: Mathematical Modelling and Numerical Analysis},
      volume={52},
      number={1},
       pages={123\ndash 145},
}

\bib{Bur13}{article}{
      author={Burman, Erik},
       title={Stabilized finite element methods for nonsymmetric, noncoercive,
  and ill-posed problems. {P}art {I}: {E}lliptic equations},
        date={2013},
        ISSN={1064-8275},
     journal={SIAM J. Sci. Comput.},
      volume={35},
      number={6},
       pages={A2752\ndash A2780},
         url={https://doi.org/10.1137/130916862},
      review={\MR{3134434}},
}

\bib{Bur14}{article}{
      author={Burman, Erik},
       title={Error estimates for stabilized finite element methods applied to
  ill-posed problems},
        date={2014},
        ISSN={1631-073X},
     journal={C. R. Math. Acad. Sci. Paris},
      volume={352},
      number={7-8},
       pages={655\ndash 659},
         url={https://doi.org/10.1016/j.crma.2014.06.008},
      review={\MR{3237821}},
}

\bib{Bur16}{incollection}{
      author={Burman, Erik},
       title={Stabilised finite element methods for ill-posed problems with
  conditional stability},
        date={2016},
   booktitle={Building bridges: connections and challenges in modern approaches
  to numerical partial differential equations},
      series={Lect. Notes Comput. Sci. Eng.},
      volume={114},
   publisher={Springer, [Cham]},
       pages={93\ndash 127},
      review={\MR{3585788}},
}

\bib{burman2017stabilized}{article}{
      author={Burman, Erik},
       title={A stabilized nonconforming finite element method for the elliptic
  cauchy problem},
        date={2017},
     journal={Mathematics of Computation},
      volume={86},
      number={303},
       pages={75\ndash 96},
}

\bib{BDE21}{article}{
      author={Burman, Erik},
      author={Delay, Guillaume},
      author={Ern, Alexandre},
       title={A hybridized high-order method for unique continuation subject to
  the {H}elmholtz equation},
        date={2021},
        ISSN={0036-1429},
     journal={SIAM J. Numer. Anal.},
      volume={59},
      number={5},
       pages={2368\ndash 2392},
         url={https://doi.org/10.1137/20M1375619},
      review={\MR{4311474}},
}

\bib{burman2018solving}{article}{
      author={Burman, Erik},
      author={Hansbo, Peter},
      author={Larson, Mats~G},
       title={Solving ill-posed control problems by stabilized finite element
  methods: an alternative to {Tikhonov} regularization},
        date={2018},
     journal={Inverse Problems},
      volume={34},
      number={3},
       pages={035004},
}

\bib{BL24}{misc}{
      author={Burman, Erik},
      author={Lu, Mingfei},
      author={Oksanen, Lauri},
       title={Solving the unique continuation problem for schr\"odinger
  equations with low regularity solutions using a stabilized finite element
  method},
        date={2024},
         url={https://arxiv.org/abs/2403.16914},
}

\bib{burman2019unique}{article}{
      author={Burman, Erik},
      author={Nechita, Mihai},
      author={Oksanen, Lauri},
       title={Unique continuation for the {Helmholtz} equation using stabilized
  finite element methods},
        date={2019},
     journal={Journal de Math{\'e}matiques Pures et Appliqu{\'e}es},
      volume={129},
       pages={1\ndash 22},
}

\bib{burman2020stabilized}{article}{
      author={Burman, Erik},
      author={Nechita, Mihai},
      author={Oksanen, Lauri},
       title={A stabilized finite element method for inverse problems subject
  to the convection--diffusion equation. {I}: diffusion-dominated regime},
        date={2020},
     journal={Numerische Mathematik},
      volume={144},
      number={3},
       pages={451\ndash 477},
}

\bib{BNO24}{article}{
      author={Burman, Erik},
      author={Nechita, Mihai},
      author={Oksanen, Lauri},
       title={Optimal approximation of unique continuation},
        date={2024},
     journal={Foundations of Computational Mathematics},
}

\bib{carstensen2022critical}{article}{
      author={Carstensen, Carsten},
      author={Sauter, Stefan},
       title={Critical functions and inf-sup stability of {Crouzeix-Raviart}
  elements},
        date={2022},
     journal={Computers \& Mathematics with Applications},
      volume={108},
       pages={12\ndash 23},
}

\bib{CKM22}{article}{
      author={Chung, Eric},
      author={Ito, Kazufumi},
      author={Yamamoto, Masahiro},
       title={Least squares formulation for ill-posed inverse problems and
  applications},
        date={2022},
        ISSN={0003-6811},
     journal={Appl. Anal.},
      volume={101},
      number={15},
       pages={5247\ndash 5261},
         url={https://doi.org/10.1080/00036811.2021.1884228},
      review={\MR{4477811}},
}

\bib{ciarlet2013analysis}{article}{
      author={Ciarlet, P., Jr.},
       title={Analysis of the {S}cott-{Z}hang interpolation in the fractional
  order {S}obolev spaces},
        date={2013},
        ISSN={1570-2820,1569-3953},
     journal={J. Numer. Math.},
      volume={21},
      number={3},
       pages={173\ndash 180},
         url={https://doi.org/10.1515/jnum-2013-0007},
      review={\MR{3118443}},
}

\bib{ciarlet2018family}{article}{
      author={Ciarlet~Jr, Patrick},
      author={Dunkl, Charles~F},
      author={Sauter, Stefan~A},
       title={A family of {Crouzeix--Raviart} finite elements in {3D}},
        date={2018},
     journal={Analysis and Applications},
      volume={16},
      number={05},
       pages={649\ndash 691},
}

\bib{crouzeix1973conforming}{article}{
      author={Crouzeix, Michel},
      author={Raviart, P-A},
       title={Conforming and nonconforming finite element methods for solving
  the stationary {Stokes} equations i},
        date={1973},
     journal={Revue fran{\c{c}}aise d'automatique informatique recherche
  op{\'e}rationnelle. Math{\'e}matique},
      volume={7},
      number={R3},
       pages={33\ndash 75},
}

\bib{DMS23}{article}{
      author={Dahmen, Wolfgang},
      author={Monsuur, Harald},
      author={Stevenson, Rob},
       title={Least squares solvers for ill-posed {PDE}s that are conditionally
  stable},
        date={2023},
        ISSN={2822-7840},
     journal={ESAIM Math. Model. Numer. Anal.},
      volume={57},
      number={4},
       pages={2227\ndash 2255},
         url={https://doi.org/10.1051/m2an/2023050},
      review={\MR{4609875}},
}

\bib{di2011mathematical}{book}{
      author={Di~Pietro, Daniele~Antonio},
      author={Ern, Alexandre},
       title={Mathematical aspects of discontinuous {Galerkin} methods},
   publisher={Springer Science \& Business Media},
        date={2011},
      volume={69},
}

\bib{ern2006evaluation}{article}{
      author={Ern, Alexandre},
      author={Guermond, Jean-Luc},
       title={Evaluation of the condition number in linear systems arising in
  finite element approximations},
        date={2006},
     journal={ESAIM: Mathematical Modelling and Numerical Analysis},
      volume={40},
      number={1},
       pages={29\ndash 48},
}

\bib{ern2017finite}{article}{
      author={Ern, Alexandre},
      author={Guermond, Jean-Luc},
       title={Finite element quasi-interpolation and best approximation},
        date={2017},
     journal={ESAIM: Mathematical Modelling and Numerical Analysis},
      volume={51},
      number={4},
       pages={1367\ndash 1385},
}

\bib{ern2021finite2}{book}{
      author={Ern, Alexandre},
      author={Guermond, Jean-Luc},
      author={others},
       title={Finite elements {II}},
   publisher={Springer},
        date={2021},
}

\bib{hadamard2003lectures}{book}{
      author={Hadamard, Jacques},
       title={Lectures on {Cauchy's} problem in linear partial differential
  equations},
   publisher={Courier Corporation},
        date={2003},
}

\bib{monsuur2024ultra}{article}{
      author={Monsuur, Harald},
      author={Stevenson, Rob},
       title={Ultra-weak least squares discretizations for unique continuation
  and {Cauchy} problems},
        date={2024},
     journal={arXiv preprint arXiv:2407.04571},
}

\bib{Wang20}{article}{
      author={Wang, Chunmei},
       title={A new primal-dual weak {G}alerkin finite element method for
  ill-posed elliptic {C}auchy problems},
        date={2020},
        ISSN={0377-0427},
     journal={J. Comput. Appl. Math.},
      volume={371},
       pages={112629, 18},
         url={https://doi.org/10.1016/j.cam.2019.112629},
      review={\MR{4057585}},
}

\bib{WW20}{article}{
      author={Wang, Chunmei},
      author={Wang, Junping},
       title={Primal-dual weak {G}alerkin finite element methods for elliptic
  {C}auchy problems},
        date={2020},
        ISSN={0898-1221},
     journal={Comput. Math. Appl.},
      volume={79},
      number={3},
       pages={746\ndash 763},
         url={https://doi.org/10.1016/j.camwa.2019.07.031},
      review={\MR{4049376}},
}

\end{biblist}
\end{bibdiv}

\end{document}